\documentclass[10pt]{amsart}
\usepackage{latexsym,amsmath,amssymb,epsfig,amsthm}
\usepackage{bm}
\usepackage{comment}
\usepackage{multirow}
\usepackage{color,mathrsfs,accents}

\newtheorem{theorem}{Theorem}
\newtheorem{lemma}{Lemma}

\newtheorem{example}{Example}
\theoremstyle{definition}
\newtheorem{remark}{Remark}
\newtheorem{proposition}{Proposition}

\usepackage{comment}
\usepackage{tabularx}
\usepackage{algorithm}
\usepackage{algorithmic}

\usepackage{dcolumn}
\newcolumntype{d}{D{.}{.}{-1}}

\newcolumntype{Y}{>{\small\center\arraybackslash}X}

\usepackage{subfigure}
\usepackage{comment}

\usepackage{multirow}

\newcommand{\beq}{\begin{equation}}
\newcommand{\eeq}{\end{equation}}
\newcommand{\be}{\begin{eqnarray}}
\newcommand{\ee}{\end{eqnarray}}
\newcommand{\bex}{\begin{eqnarray*}}
\newcommand{\eex}{\end{eqnarray*}}
\newcommand{\ba}{\begin{array}}
\newcommand{\ea}{\end{array}}

\font\tenbi=cmmib10   at 11 pt
\font\sevenbi=cmmib10 at 9pt
\font\fivebi=cmmib7 at 6pt
\newfam\bifam
\textfont\bifam=\tenbi \scriptfont\bifam=\sevenbi
\scriptscriptfont\bifam=\fivebi

\def\bq{\begin{equation}}
\def\eq{\end{equation}}
\def\br{\begin{eqnarray}}
\def\er{\end{eqnarray}}
\def\brr{\bq\begin{array}{rlll}}
\def\err{\end{array}\eq}

\def\OH{\accentset{\circ}H}
 
 \def\i{\operatorname{i}}
 \def\e{\operatorname{e}}
 
 \def\RR{\mathbb{R}}
 \def\NN{\mathbb{N}}

\graphicspath{{figures/}}
\begin{document}

\title[Spectral method for  Riesz fractional eigenvalue problems]{Jacobi-Galerkin spectral method for  eigenvalue problems  of Riesz fractional differential equations}

\author[
	L. Chen,\; Z. Mao,\; $\&$\; H. Li
	]{
		\;\; Lizhen Chen${}^1$,    \;\; Zhiping Mao${}^{2}$ \;\;  and  \;\; Huiyuan Li${}^{3}$ }
		
	\thanks{${}^1$Beijing Computational Sciences and Research Center, Beijing 100193, China. Email: lzchen@csrc.ac.cn. The research of this author is partially supported by the National Natural Science Foundation
of China (NSFC, 11671166),  the Postdoctoral Science Foundation
of China  (2015M580038) and the Joint Fund of the National Natural Science Foundation of China and the China Academy of Engineering Physics (NSAF  U1530401).\\
		\indent ${}^{2}$Division of Applied Mathematics, Brown University, 182 George St., Providence RI 02912, USA. Email: zhiping\_mao@brown.edu. The research of this author is partially supported by MURI/ARO (W911NF-15-1-0562) on ``Fractional PDEs for Conservation Laws and Beyond: Theory, Numerics and Applications".\\
		\indent ${}^{3}$State Key Laboratory of Computer Science/Laboratory of Parallel Computing,  Institute of Software, Chinese Academy of Sciences, Beijing 100190, China.  Email: huiyuan@iscas.ac.cn.  The research of this author is partially   supported by the National Natural Science Foundation of China (NSFC 11471312, 91430216 and 91130014).
}

%

\date{}



\begin{abstract}
An efficient Jacobi-Galerkin spectral method for  calculating eigenvalues of Riesz fractional partial differential equations with homogeneous Dirichlet boundary values is proposed in this paper. In order to retain the symmetry and positive definiteness  of the discrete linear system,  we introduce some properly defined Sobolev spaces and  approximate  the eigenvalue problem in a standard
 Galerkin weak formulation instead of the
 Petrov-Galerkin one as in literature. Poincar\'{e} and inverse inequalities are proved for the proposed Galerkin formulation which finally help us  establishing  a sharp estimate on the algebraic system's condition number.  Rigorous error estimates of the eigenvalues and eigenvectors are then readily obtained by using  Babu\v{s}ka and Osborn's  approximation theory on self-adjoint and positive-definite  eigenvalue problems. Numerical results are presented to demonstrate the accuracy and efficiency,
and to validate the asymptotically exponential oder of convergence. Moreover, 
the Weyl-type asymptotic law $ \lambda_n=\mathcal{O}(n^{2\alpha})$
for  the $n$-th eigenvalue $\lambda_n$ of the Riesz fractional differential operator  of order $2\alpha$, and 
the condition number $N^{4\alpha}$ of its algebraic system  with respect to the polynomial degree $N$ are observed.


\end{abstract}

\maketitle

\section{Introduction}
\label{sec:intro}

Fractional  differential equations (FDEs) have been widely used in modeling of many nonlocal phenomena
arising from science and engineering, such as viscoelasticity,
electromagnetism and so on (see \cite{KST06,MR93,P99,BD03,BD96,E98} and the references therein),
among which the Riesz fractional  differential equations are of common interests
of mathematicians and physicists. It is widely  assumed that Riesz fractional derivatives are  equivalent to  fractional Laplacians   \cite{CS07}.
Riesz FDEs have been numerically studied extensively,  including, finite difference method \cite{DD12},
 a fourth order compact alternating direction implicit (ADI) scheme \cite{ZSH14}, finite element methods \cite{BTY14,ZXX13,EHR07,R06} and spectral methods \cite{ZLLBTA14,MCS16}.
 In particular, Mao et al. have developed recently a novel spectral Petrov-Galerkin  method to solve the boundary value problem of  Riesz fractional differential equations and established the error estimate in non-uniformly weighted Sobolev spaces \cite{MCS16}. 

 The eigenvalue problem for the Riesz fractional differential euqations
 is  challenging and  has  attracted lots of attention
(see \cite{K12,BP14,PRS16,WZ13,ZLWW14,WZ12,SZHL13,DWLQ13,NP14,LQ15,ZRK07}).  Most of the existing studies  focus on the
theoretical research.
Let $\Lambda=(-1,1)$ and $2\alpha \in (0,2)$.
Kwa\'snicki \cite{K12} introduced the Weyl-type asymptotic law for the eigenvalues of the one-dimensional fractional Laplace operator
$(-\Delta)^{\alpha}$  on the interval $\Lambda$ with the zero exterior boundary conditions: the $n$-th eigenvalue is equal to $(n\pi/2 -
(2 - 2\alpha)\pi/8)^{2\alpha} + \mathcal{O}(1/n)$. DeBlassie \cite{D04} and Chen
et al. \cite{CS05} have derived the estimate for $\lambda_n$ is $\displaystyle \frac{1}{2}(\frac{n\pi}{2})^{2\alpha} \leq \lambda_n \leq(\frac{n\pi}{2})^{2\alpha}$.
Meanwhile, owing to the non-locality of Reisz fractional derivatives, it is usually impossible to  obtain analytically  a closed expression  for the  eigenfunctions,  and  it is also
hard to  precisely specify the  (asymptotic)  behavior  of an eigenfunction $\psi_k(x)$ around the endpoints $x=\pm1$.  This motivates   researchers to carry out numerical studies on  the Riesz  differential  eigenvalue problem.  Zoia et al. \cite{ZRK07} have provided a discretized version of the Riesz differential operator  of order $2\alpha$. The eigenvalues and eigenfunctions in a bounded domain were then gained numerically for different boundary conditions. However,  the eigenfunctions of a  Riesz fractional differential operator have only  a limited regularity measured in usual Sobolev space, and  eigensolutions  obtained
by ordinary numerical methods 
have a very poor  accuracy.
Based on these numerical  eigensolutions,
analytical results  concerning the spectrum structure   may not be reliable or even  be   erroneous  in some cases \cite{ZRK07}.
Indeed, Borthagaray et al. \cite{Borthagaray} have shown that the first
eigenfunction $\psi_1$ belongs to $H^{\alpha+1/2-\varepsilon}(\Lambda)$ for any $\varepsilon>0$, and the conforming finite element
method exhibits a convergence rate of order  $1-\varepsilon$.
Hence, one may ask for  high order methods to conquer this difficulty.

Fortunately, evidences showed that in general for  the eigenfunctions  scale as $(1-x^2)^{\alpha}$ as $x\to \pm 1$  \cite{Buldyrev,Zumofen} , and  the first eigenfunction  $\psi_1(x)\sim \frac{\Gamma(3/2+2\alpha)}{\sqrt{\pi} \Gamma(1+2\alpha)} (1-x^2)^{\alpha} $ as $2\alpha\to +\infty$ \cite{Bottcher}.
 Using
the Jacobi  basis functions $\{(1-x^2)^{\alpha} J^{\alpha,\alpha}_n(x)\}_{n=0}^N$
to mimic the singular behavior of the eigenfunctions,   one may  naturally  expect
a spectrally   high order  of convergence rate  of
the  specifically  devised  spectral method   for the eigenvalue problems.

The main purpose of this paper is thus  to propose an efficient Jacobi-Galerkin spectral method for solving
Riesz fractional differential eigenvalue problems, and then to conduct a comprehensive numerical
analysis.

We start with a brief review  over some definitions of  Riesz fractional derivatives on the interval $\Lambda$.  Some Sobolev spaces are then introduced  to fit the domain of Riesz fractional derivatives, in which the Riesz fractional derivatives are proved to be self-adjoint and positive definite.
Thus, the Riesz fractional differential eigenvalue problems can be equivalently  written  in a symmetric
weak formulation.
This recognition eventually helps us  to propose a Jacobi-Galerkin spectral method,  instead of the Petrov-Galerkin one, for  solving
Riesz fractional differential eigenvalue problems. Moreover,
 by adopting the  generalized Jacobi functions as the basis functions, an efficient  implementation
  is given for  this Jacobi-Galerkin spectral method. Indeed, the stiffness matrix is identity owing to  orthonormality of the basis function with  respect to the inner product induced by Riesz fractional operators; while all entries of the mass matrix can be explicitly
 evaluated via their analytical formula.

 The symmetric variational  formulation also play an important role  in the numerical analysis.   On the one hand, the Poincar\'{e}   inequality and the inverse inequality  for Riesz fractional derivatives  are derived,  through  which the lower and  the upper bounds of  all eigenvalues are established, respectively. Moreover, an elaborative analysis shows that the smallest  numerical  eigenvalue  behaves as $\mathcal{O}(1)$  while the largest one
 of the Riesz fractional  differential operator  of order $2\alpha$   behaves as $\mathcal{O}(N^{4\alpha})$. This indicates that  the condition number of the mass matrix
increases  at a rate of $\mathcal{O}(N^{4\alpha})$.  On the other hand,
following the approximation theory of  Babu\v{s}ka and Osborn on the Ritz method for self-adjoint and positive-definite  eigenvalue problems,
 rigorous error estimates
for both the eigenvalues and eigenfunctions are presented.

Numerical experiments demonstrate that our Jacobi-Galerkin method
has a higher accuracy than all existing methods ever known.
 Indeed,
asymptotically exponential (sub-geometric) orders of convergence  are observed
for large $N$ in the  error plots in  both semi-logarithm and
logarithm-logarithm scales. This hypothesis  is also convinced by
 a reliability analysis of all computational eigenvalues
together with a continuity analysis of the first eigenvalue upon the fractional order $2\alpha$. Besides,
the order of $\mathcal{O}(N^{4\alpha})$ for the condition number of the reduced linear algebraic eigen-system
(also of the mass matrix) is also confirmed numerically.

The remainder of this paper is organized as follows.  We describe some notations and preliminary results
about the fractional derivatives operators in Section \ref{sec:problem}.
In Section \ref{sec:imp}, we propose an efficient  Jacobi-Galerkin spectral method
 for the Riesz fractional differential eigenvalue
problem with homogeneous Dirichlet boundary conditions. The detail of the numerical implementation is also described. Furthermore,  the rigorous error estimate of the valid eigenvalues and eigenvectors is derived by the approximation theory of Babu\v{s}ka and Osborn on
the Ritz method for self-adjoint and positive-definite  eigenvalue problems.
 Some numerical results  which verify  the  efficiency and accuracy of the Jacobi-Galerkin spectral method are provided
in Section \ref{sec:results}.
Finally,  a conclusion remark  is made  in the last section.

\section{Preliminaries} \label{sec:problem}

 Throughout this paper, we shall use the notations $a\lesssim b$ for $a\le C b$, $a\gtrsim b$ for $a\ge c b$, and
$a \asymp b$ for $c  b \le  a \le C b$  with some generic positive constants $c$ and $C$ which are independent of
any function and of any discretization parameters.
Next denote by
$\mathbb{N}$ (resp. $\mathbb{R}$)  the set of positive integers (resp. real numbers). Further denote $\mathbb{N}_0:=\{0\}\cup \mathbb{N}$.

Let
$\omega$ be a generic positive
weight function  which is not necessary in $
L^1(I)$ ($I\subseteq \RR$). Denote by $(u,v)_{\omega,I}:=\int_{I} u(x) \overline{v(x)}
\omega(x) d x $ the inner product of $L^2_\omega(I)$ with the norm
$\|\cdot\|_{\omega,I}$.  Whenever $I=\Lambda:=(-1,1)$, the subscript $I$ will be omitted.
\subsection{Fractional integrals and derivatives} \label{sec:notation}

Furthermore, we recall the definitions of the fractional integrals and derivatives in the sense of
Riemann-Liouville.

$\mathbf{Definition~1}$ (Fractional integrals and derivatives). For $\rho >0$, the
left and right fractional integral are  defined respectively as
\begin{align*}
{}_{-1}I^{\rho}_x v(x):&= \displaystyle \frac{1}{\Gamma(\rho)} \int \limits_{-1}^x \frac{v(t)}{(x-t)^{1-\rho}}dt, \quad x>-1, \\
{}_{x}I^{\rho}_1 v(x):&= \displaystyle \frac{1}{\Gamma(\rho)} \int \limits^{1}_x \frac{v(t)}{(t-x)^{1-\rho}}dt, \quad  x<1.
\end{align*}
For real $s \in [k-1,k)$ with $k\in \mathbb{N}$, the left-side and the right-side Riemann-Liouville fractional
derivative (LRLFD and RRLFD) of order $s$ are defined respectively by
\begin{align*}
{}_{-1}D^{s}_x v(x):&= \displaystyle \frac{1}{\Gamma(k-s)} \frac{d^k}{dx^k}\int \limits_{-1}^x \frac{v(t)}{(x-t)^{s-k+1}}dt, x \in \Lambda, \\
{}_{x}D^{s}_1 v(x):&= \displaystyle \frac{(-1)^k}{\Gamma(k-s)}\frac{d^k}{dx^k} \int \limits^{1}_x \frac{v(t)}{(t-x)^{s-k+1}}dt, x\in \Lambda.
\end{align*}

$\mathbf{Definition~2}$ (Fractional Riesz integral and derivatives). For $\rho >0$, the
Riesz fractional integrals are defined as
\begin{align*}
I_1^{\rho} v(x):&= \displaystyle \frac{1}{2\Gamma(\rho)\sin\frac{\pi \rho } 2} \int \limits_{-1}^1 \frac{ \mbox{sgn}(x-t)}{|x-t|^{1-\rho}}v(t)dt
=\frac{1}{2\sin\frac{\pi \rho } 2}({}_{-1}I^{\rho}_x-{}_{x}I^{\rho}_1 )v(x), &&x\in \Lambda,\, \frac{\rho}{2}\not \in \mathbb{N}, \\
I_2^{\rho} v(x):&= \displaystyle \frac{1}{2\Gamma(\rho)\cos\frac{\pi \rho } 2} \int \limits_{-1}^1 \frac{ v(t)}{|x-t|^{1-\rho}}dt
=\frac{1}{2\cos\frac{\pi \rho } 2}({}_{-1}I^{\rho}_x+{}_{x}I^{\rho}_1 )v(x), &&x\in \Lambda,\,
 \frac{\rho+1}{2}\not \in \mathbb{N},
\end{align*}
where sgn  is the sign function. 

For $s\in [k-1,k)$ with $k\in \NN$,
 we define the Riesz fractional derivative (RFD)
of order $s$:
\begin{equation}
\label{Def-RFD}
D^{s}v(x):= \frac{ (-1)^{\lfloor \frac{s+1}{2} \rfloor} }{ 2\cos \frac{\pi s}{2}} \big[ {}_{-1}D_x^s + {}_xD_1^s \big] v(x)
=
              \begin{cases}
                \dfrac{d^{k}}{dx^{k}}I_1^{k-s} v(x), & k \text{ odd}, \\[0.5em]
                \dfrac{d^{k}}{dx^{k}}I_2^{k-s}v(x), & k  \text{ even},
              \end{cases}
            \qquad x\in \Lambda,
\end{equation}
hereafter $\lfloor a \rfloor$ is the greatest integer not exceeding  the real number $a$.
Whenever $s$ is an  odd integer, the definition of $D^sv$
is  then treated in the limit sense.

Let us
define the Fourier transform of a function $u\in L^2(\RR)$,
\begin{align*}
    \widehat v(\xi) = [\mathscr{F}v](\xi)=\int_{\RR} v(x) \e^{-\i\xi x} dx, \quad \xi\in \RR.
\end{align*}
Further let $\tilde{v}$ be the zero extension of $v$,
\begin{align}
\label{Zext}
    \tilde{v}\big|_{\Lambda} = v \quad \text{ and } \quad
\tilde{v}\big|_{\RR\setminus \Lambda} = 0.
\end{align}
Then one finds the  following alternative definition of the Riesz fractional integral and derivative \cite{SKM93},
\begin{align}
&[\mathscr{F}(I_1^{\rho}v)](\xi) = -\i \mathrm{sgn}(\xi)\, |\xi|^{-\rho}\, [\mathscr{F}\tilde v](\xi),
\qquad\qquad  [\mathscr{F}(I_2^{\rho}v)](\xi) = |\xi|^{-\rho}\, [\mathscr{F}\tilde v](\xi),
\notag
\\
\label{F-RFD}
&[\mathscr{F}(D^sv)](\xi) =(-1)^{\lfloor \frac{s+1}{2} \rfloor}  |\xi|^{s}[\mathscr{F}\tilde v](\xi)= \begin{cases}
(-1)^\frac{k}2 |\xi|^{s}[\mathscr{F}\tilde v](\xi), & k \text{ even},
\\
  (-1)^\frac{k-1}2 |\xi|^{s}[\mathscr{F}\tilde v](\xi), & k \text{ odd},
  \end{cases}
  \qquad v\in\mathscr{D}(\Lambda),
\end{align}
which implies that, for any $n\in \NN_0$, $D^{2n}=(-1)^n\dfrac{d^{2n}}{dx^{2n}}$,  while
$D^{2n+1}$ is totally distinct from   $\dfrac{d^{2n+1}}{dx^{2n+1}}$.

\begin{remark}
We would like  to emphasize   the equivalence of  Riesz fractional derivative and  the fractional Laplacian.
Indeed, among several ways to define the fractional Laplacian $(-\Delta)^{s/2} $ on the (bounded) interval $\Lambda$,   the one using the zero extension and the pseudo-differential operator of symbol $|\xi|^{s}$ is of our greatest interest \cite{ACOSTA17},
\begin{align}
\label{F-FL}
(-\Delta)^{s/2} v(x) = [\mathscr{F}^{-1}( |\xi|^{s} \mathscr{F}\tilde v )](x),
\end{align}
which reveals that the Riesz fractional derivative \eqref{Def-RFD}  on $\Lambda$ coincides, up to  a constant  $\pm 1$,   with the fractional Laplacian \eqref{F-FL}.
\end{remark}

\subsection{Sobolev spaces }

It is well known that $H^s(\RR)$ can  be defined through the Fourier transform \cite{Lions72, Triebel, Grisvard},
\begin{align}
\begin{split}
  &H^s(\RR)=\big\{v\in \mathscr{D}'(\RR):   \|v\|_{s,\RR}<\infty \big\},
\\
& |v|_{s,\RR} =  \Big( \frac{1}{2\pi}\int_{\RR}\  {|\xi|}^{2s}
\big|\widehat{ v}({\xi})\big|^2 d{\xi}\Big)^{\frac12},
 \qquad\qquad   \|v\|_{s,\RR} =  \Big( \frac{1}{2\pi}\int_{\RR}\big(1+ |\xi|^{2}\big)^s
\big|\widehat{ v}({\xi})\big|^2 d{\xi}\Big)^{\frac12};
\end{split}
\end{align}
and   $H^s(\Lambda)$  ($s\ge 0$)  can be derived
from $H^s(\RR)$ by extension \cite{Grisvard,Lions72, Triebel},
\begin{align}
  H^s(\Lambda)= \big\{v = u|_{\Lambda}: u\in H^s(\RR) \big\}, &&
  \|v\|_{s} = \inf_{u\in H^s(\RR)  \atop u|_{\Lambda}=v }\|u\|_{s,\RR}.
  \label{extension}
\end{align}
The completion of $\mathscr{D}(\Lambda)$ in  $H^s(\Lambda)$ is denoted by  $H^s_0( \Lambda)$.

The zero extension $\tilde v$ of $v$ as defined in \eqref{Zext}
is of  particular interest  \cite{Grisvard,Lions72}.
By Theorem 11.4 of \cite{Lions72},  $v\mapsto \tilde v$ is a continuous mapping of $H^s(\Lambda)\mapsto H^s(\RR)$ if and only if $0\le s<\frac12$,
is  a  continuous mapping of $H_0^s(\Lambda)\mapsto H^s(\RR)$ for $s>\frac12$ if and only if $s+\frac12\not\in \NN$.

The zero extension  induces  a special type of  Sobolev spaces,
\begin{align*}
    H^s_*(\Lambda) = \left\{v \in \mathscr{D}'(\Lambda): \tilde v \in H^s(\RR)  \right\},
    \qquad \|v\|_{s,*} = \|\tilde v\|_{s,\RR}, \   |v|_{s,*}  = |\tilde v|_{s,\RR}.
\end{align*}
Actually,  $\mathscr{D}(\Lambda)$ is  dense  in $H^s_*(\Lambda)$ for  $s\ge 0$ \cite[Theorem 3.2.4/1]{Triebel}, thus  the completion of $\mathscr{D}(\Lambda)$  in $H^s_{*}(\Lambda)$ is $H^s_{*}(\Lambda)$ itself,
i.e., space $H^s_{*0}(\Lambda)=H^s_{*}(\Lambda)$.
More precisely,
  it holds that
\begin{align*}
 H^s_{*0}(\Lambda) =H^s_*(\Lambda) =\begin{cases} H^{s}(\Lambda)=H^{s}_0(\Lambda)
, & \text{ if } 0< s<\frac12,\\
     H^{s}_{00}(\Lambda)\subset H^{s}_{0}(\Lambda), & \text{ if } s+\frac12 \in \NN,\\
     H^{s}_0(\Lambda) \subset H^{s}(\Lambda), & \text{ if } s>\frac12,\ s+\frac12\not \in \NN,
\end{cases}\end{align*}
where the interpolation space $H^{m+\frac12}_{00}(\Lambda) = [H^{m}_{0}(\Lambda),H^{m+1}_{0}(\Lambda)]_{\frac12}$.

Indeed, according to \cite{Lions72}, $(-1)^m D^{2m}$ ($m\in \NN$) is self-adjoint, positive and unbounded in $H_0^{2m}(\Lambda)$, which is dense in $H^0(\Lambda)=L^2(\Lambda)$ with continuous injection.  Moreover, $\|v\|_{2m}^2 \asymp \|v\|^2+ \|D^{2m}v\|^2 \asymp  \|D^{2m}v\|^2 \asymp |v|_{2m}^2$ for all $v\in H^{2m}_0(\Lambda)$.
Thus by the space interpolation theory,
\begin{align}
H^{s}_*(\Lambda)=[H^0(\Lambda), H^{2m}_0(\Lambda)]_{\frac{s}{2m}}
=\left\{v\in  \mathscr{D}'(\Lambda):   \|v\|^2 + \|D^s v \|^2 <\infty\right\}.
\end{align}

For any $s\in [k-\frac12,k+\frac12)$ with the positive integer $k$, one readily  obtains from the adjoint properties of the
Riemann-Liouville fractional differential operators that \cite{ER06,LiXu09,MS16}
\begin{align}
\label{selfadjoint}
(-1)^k(D^{2s}u,v) = (D^su, D^sv) = (-1)^k(u, D^{2s}v),\qquad  u,v\in H_*^{s}(\Lambda).
\end{align}

Moreoever, by \eqref{F-RFD} together with Parserval's theorem,
\begin{align}
\label{SemiNorm}
 |v|_{s,*}  
 = \|D^sv\|:=|v|_s.
\end{align}
Indeed, by  Lemma 1.3.2.6 in \cite{Grisvard},
it holds  for $0\le s< 1$ that
\begin{align}
\label{normH00}
 \| v\|_{s,*}^2 \asymp  | v|_{s,*}^2
\asymp     |v|_{s}^2 + \|v\|_{\omega^{-2s,-2s}}^2,\qquad v\in H_*^s(\Lambda),
\end{align}
where  $\omega^{\alpha,\alpha}=\omega^{\alpha,\alpha}(x)=(1-x^2)^{\alpha}$.

We now conclude this subsection with the following embedding theorem.
\begin{theorem}
\label{embedding}
$H^s_*(\Lambda)$ with $s>0$ is compactly embedded in $L^2(\Lambda)$, i.e., $H^s_*(\Lambda)\hookrightarrow L^2(\Lambda)$.
\end{theorem}
\begin{proof}
One first notes that $H^s(\Lambda)$ ($s>0$) is compactly embedded in $L^2(\Lambda)$
and $H^s_*(\Lambda)\subseteq H^s(\Lambda)$,  then Theorem \ref{embedding}
is readily checked.
\end{proof}

\subsection{Generalized Jacobi functions}

Let
$P_n^{\alpha,\beta}(x),\, n\in \NN_0$ with $\alpha,\beta>-1$  be the classical Jacobi polynomials which are orthogonal with
respect to the weight function $\omega^{\alpha,\beta}(x)=(1-x)^{\alpha}(1+x)^{\beta}$
over $\Lambda$, i.e.
\begin{equation}
\label{OrthP}
( P_n^{\alpha,\beta}, P_m^{\alpha,\beta})_{\omega^{\alpha,\beta}}=\gamma_{n}^{\alpha,\beta}\delta_{nm},
\qquad \gamma_{n}^{\alpha,\beta} =
\frac{2^{\alpha+\beta+1}}{2n+\alpha+\beta+1}\,
\frac{\Gamma(n+\alpha+1)\Gamma(n+\beta+1)}{n!\Gamma(n+\alpha+\beta+1)},
\end{equation}
where $\delta_{nm}$ is the Dirac delta symbol.

We now define a special type of generalized Jacobi functions:
\begin{equation}
\label{JP}
\mathcal{J}_n^{-\alpha,-\alpha}(x)=(1-x^2)^{\alpha}P_n^{\alpha,\alpha}(x), \qquad
 x \in \Lambda,\, n \in \mathbb{N}_0,\,  \alpha >-1.
\end{equation}
  It is clear that  $\mathcal{J}_n^{-\alpha,-\alpha}(x), n\in \NN_0$,
satisfy the homogeneous boundary conditions
\begin{equation}
\label{HomoBnd}
\partial_x^l\mathcal{J}_n^{-\alpha,-\alpha}(\pm 1)=0 ,\quad  l=0,1,\cdots,\lceil\alpha\rceil-1,\end{equation}
where $\lceil\alpha\rceil$ stands for the smallest integer not less than  $\alpha.$
\begin{lemma}
If $\alpha>0$, $2\alpha\in [2k-1,2k+1)$   and $\nu\in \{ 0,1, \dots, \lfloor \alpha\rfloor\}$, then
\begin{equation} \label{derpol}
D^{2\alpha-2\nu} \mathcal{J}_n^{-\alpha,-\alpha}(x)= (-1)^{k}\dfrac{2^{2\nu} \Gamma(n+2\alpha-2\nu+1)}{n!} P^{\alpha-2\nu,\alpha-2\nu}_{n+2\nu}(x), \quad n\in \NN_0.
\end{equation}
\end{lemma}
\begin{proof}
We  resort to the following identities on the Riesz integral of  Jacobi functions
\cite[Theorem 6.5 and Theorem 6.7]{P99},
\begin{align}
\label{I2}
\begin{split}
\int_{-1}^1  \frac{\omega^{ \frac{s-1}{2}+i , \frac{s-1}{2}+j } (t)
P^{ \frac{s-1}{2}+i , \frac{s-1}{2}+j }_n (t) }{ |x-t|^{s}} dt
= \frac{\pi (-1)^i 2^{i+j} \Gamma(n+s)}{n! \Gamma(s) \cos \frac{s \pi}{2}}
P^{\frac{s-1}{2}-i,\frac{s-1}{2}-j}_{n+i+j}(x),
\\
 x\in (-1,1),\  s\in (0,1),\  i,j\in \mathbb{Z}, \  i,j>-\frac{s+1}{2},\  n+i+j\ge 0;
 \end{split}
 \\
\label{I1}
 \begin{split}
\int_{-1}^1  \frac{\mathrm{sgn}(x-t) }{ |x-t|^{s}} \omega^{ \frac{s}{2}+i , \frac{s}{2}+j } (t)
P^{ \frac{s}{2}+i , \frac{s}{2}+j }_n (t) dt
= \frac{\pi (-1)^i 2^{i+j+1} \Gamma(n+s)}{n! \Gamma(s) \sin \frac{s \pi}{2}}
P^{\frac{s}{2}-i-1,\frac{s}{2}-j-1}_{n+i+j+1}(x),
\\
 x\in (-1,1),\  s\in (0,1),\  i,j\in \mathbb{Z}, \  i,j>-\frac{s+2}{2},\  n+i+j+1\ge 0.
 \end{split}
\end{align}
Thus, by taking  $s=2\alpha-2k$ and $i =j= k$  when  $2\alpha\in [2k,2k+1)$, we  derive that
\begin{align*}
&I_{1}^{2k+1-2\alpha}\mathcal{J}^{-\alpha,-\alpha}_n(x)
=
\frac{ (-1)^k 2^{2k+1} \Gamma(n+2\alpha-2k)}{n!  }
P^{\alpha-2k-1,\alpha-2k-1}_{n+2k+1}(x),
\\
&D^{2\alpha-2\nu}\mathcal{J}^{-\alpha,-\alpha}_n(x)  =
\frac{ (-1)^k 2^{2k+1} \Gamma(n+2\alpha-2k)}{n!  }
\frac{d^{2k+1-2\nu}}{dx^{2k+1-2\nu}}P^{\alpha-2k-1,\alpha-2k-1}_{n+2k+1}(x)
\\
 &\qquad =
\frac{ (-1)^k 2^{2\nu} \Gamma(n+2\alpha-2\nu+1)}{n!  }
P^{\alpha-2\nu,\alpha-2\nu}_{n+2\nu}(x);
\end{align*}
While, by taking  $s=2\alpha-2k+1$ and $i =j= k$  when  $2\alpha\in [2k-1,2k)$, we also get that
\begin{align*}
&I_{2}^{2k-2\alpha}\mathcal{J}^{-\alpha,-\alpha}_n(x)
=
 \frac{ (-1)^k 2^{2k} \Gamma(n+2\alpha-2k+1)}{n! }
P^{\alpha-2k,\alpha-2k}_{n+2k}(x),
\\
&D^{2\alpha-2\nu}\mathcal{J}^{-\alpha,-\alpha}_n(x)  =
 \frac{ (-1)^k 2^{2k} \Gamma(n+2\alpha-2k+1)}{n! }
\frac{d^{2k-2\nu}}{dx^{2k-2\nu}}P^{\alpha-2k,\alpha-2k}_{n+2k}(x)
\\
& \qquad =
\frac{ (-1)^k 2^{2\nu} \Gamma(n+2\alpha-2\nu+1)}{n!  }
P^{\alpha-2\nu,\alpha-2\nu}_{n+2\nu}(x).
\end{align*}
The proof is now completed.
\end{proof}

By \eqref{selfadjoint}, \eqref{derpol} and \eqref{OrthP}, one arrives at the following theorem.

\begin{theorem}
\label{TH:Basis}
 $\mathcal{J}_m^{-\alpha,-\alpha},\ m\ge 0$ form a complete  orthogonal
system in $H_*^{\alpha}(\Lambda)$. More precisely,
\begin{align}
\label{OrthJ}
 (D^{\alpha}\mathcal{J}_m^{-\alpha,-\alpha}, D^{\alpha}\mathcal{J}_n^{-\alpha,-\alpha} )
 = \frac{2^{2\alpha+1}\Gamma(m+\alpha+1)^2}{m!^2(2m+2\alpha+1)} \delta_{m,n},
 \qquad m,n\in \NN_0,\, \alpha\ge 0.
\end{align}
\end{theorem}

As an immediate consequence of Theorem \ref{TH:Basis},  any $u\in H_*^s(\Lambda)$
has an expansion in generalized Jacobi functions,
\begin{align}
u(x) =  \sum_{i=0}^{\infty} u_i \mathcal{J}^{-s,-s}_i(x), \qquad
  u_i  = \frac{(D^su, D^s\mathcal{J}^{-s,-s}_i)}{(D^s\mathcal{J}^{-s,-s}_i, D^s\mathcal{J}^{-s,-s}_i)},
\end{align}
and
\begin{align}
\label{Parservalu}
|u|_s^2 =  \sum_{i=0}^{\infty} \frac{2^{2\alpha+1}\Gamma(i+\alpha+1)^2}{i!^2(2i+2\alpha+1)} | u_i|^2.
\end{align}

\section{Riesz fractional differential eigenvalue problems  and Jacobi-Galerkin approximation}
\label{sec:imp}
\subsection{Riesz fractional differential eigenvalue problems}
We consider the Riesz  fractional differential eigenvalue problems
of order $2\alpha \in [2k-1,2k+1)$  with $k\in \mathbb{N}$,
\begin{equation} \label{eigp1}
\begin{cases}
(-1)^{k}D^{2\alpha} u(x)=\lambda  u(x),   & x\in\Lambda,
\\
u^{(l)}(\pm 1)=0 ,&   l=0,1,\cdots,k-1. 
\end{cases}
\end{equation}
The weak formulation of \eqref{eigp1} reads:
to find  nontrivial  $(\lambda, u) \in \mathbb{R} \times H^{\alpha}_*(\Lambda)$, such that
\begin{equation} \label{weakf}
a(u,v)=\lambda b(u,v),\qquad v \in H^{\alpha}_*(\Lambda),
\end{equation}
where $a(\cdot,\cdot)$ and $b(\cdot,\cdot)$ are the bilinear forms defined by
\begin{align*}
a(u,v)&=(D^{\alpha} u,  D^{\alpha} v),
\quad && u,v \in H^{\alpha}_*(\Lambda), \\
b(u,v)&=(u,v)
,&&  u,v \in L^2(\Lambda).
\end{align*}

It is obvious that
$a(\cdot,\cdot)$ and $b(\cdot,\cdot)$ are symmetric, positive definite, continuous
and coercive on  $H^{\alpha}_*(\Lambda)\times H^{\alpha}_*(\Lambda)$ and
$L^2(\Lambda)\times L^2(\Lambda)$, respectively.

\begin{remark}We would like to point out  that
unlike the weak formulation in  \cite{MCS16},  our weak formulation \eqref{weakf}  here
is formulated in a standard symmetrical way, which is convenient for us to
 conduct the numerical analysis on eigenvalue problems of the fractional differential equations.
\end{remark}

\begin{remark}
 Problem \eqref{eigp1} admits an infinite sequence of  eigensolutions $\{(\lambda_i, \psi_i)\}_{i=1}^{\infty}$
with the eigenvalues being ordered  increasingly,
$\lambda_1<\lambda_2\le \lambda_3\le \cdots$.
All eigenvalues  of  \eqref{eigp1} are real and positive.
Moreover, it is proved that  all eigenvalues are simple if  $\frac12 < \alpha\le  2$  \cite{K12}.
Besides the following proposition on the Poincar\'{e} inequality indicates,
$$\lambda_1>\Gamma(2\alpha+1).$$
\end{remark}

\begin{proposition} \label{poincare}
(Poincar\'{e} inequality) Suppose  $u \in H^{\alpha}_*(\Lambda)\setminus\{0\}$, then
\begin{equation}
|u|_{\alpha,*}  > \sqrt{\Gamma(2\alpha+1)}\, \|u\|_{\omega^{-\alpha,-\alpha}}> \sqrt{\Gamma(2\alpha+1)}\, \|u\|.
\end{equation}
\end{proposition}
\begin{proof} For any $ i \in \mathbb{N}_0$,
\begin{equation*}
\frac{\Gamma(i+2\alpha+1)}{i!}= \frac{(i+2\alpha)(i+2\alpha-1)
\cdots (2\alpha+1)}{i!}\, \Gamma(2\alpha+1)
>\Gamma(2\alpha+1).
\end{equation*}
Then, for any $ u \in H^{\alpha}_*(\Lambda)$ with $
u(x)=\displaystyle \sum_{i=0}^\infty u_i\mathcal{J}_i^{-\alpha,-\alpha}(x),$
 we get from \eqref{OrthJ}, \eqref{OrthP} together with the Parseval's identities that
\begin{align*}
|u|_{\alpha}^2 =&\,  \sum_{i=0}^{\infty}
\frac{\Gamma(i+2\alpha+1)}{i!} \frac{2^{2\alpha+1}\Gamma(i+\alpha+1)^2}{i!\Gamma(i+2\alpha+1)(2i+2\alpha+1)} | u_i|^2
 \\
>  &\,\Gamma(2\alpha+1) \sum_{i=0}^{\infty}  \frac{2^{2\alpha+1}\Gamma(i+\alpha+1)^2}{i!\Gamma(i+2\alpha+1)(2i+2\alpha+1)} | u_i|^2
 \\
 =&\,\Gamma(2\alpha+1)\|u\|_{\omega^{-\alpha,-\alpha}}^2 >\Gamma(2\alpha+1)\|u\|^2,
\end{align*}
which completes the proof.
\end{proof}

\subsection{Jacobi-Galerkin spectral approximation and implementation}

Let $\mathbb{P}_N$ be the set of polynomials with degree less than or equal to $N$.  Then, we define the
finite dimensional fractional space for $\alpha>0$,
\begin{equation*}
\mathbb{F}_N^{-\alpha,-\alpha}(\Lambda):=\mbox{span}\Big\{\phi_n(x):=\frac{\sqrt{2n+2\alpha+1}\,n! }{2^{\alpha+\frac12}\,\Gamma(n+\alpha+1)}\mathcal{J}_n^{-\alpha,-\alpha}(x),0\leq n \leq N\Big\}.
\end{equation*}
 The Jacobi-Galerkin  spectral method for  \eqref{weakf} can be written as follows,
to find $(\lambda_N, u_N) \in \mathbb{R} \times \mathbb{F}_N^{-\alpha,-\alpha}(\Lambda)$ such that
\begin{equation} \label{weakn}
a(u_N,v_N)=\lambda_N b(u_N,v_N), \qquad  v_N \in \mathbb{F}_N^{-\alpha,-\alpha}(\Lambda).
\end{equation}

We now  give a brief description on  the  numerical implementation of the Jacobi Galerkin spectral
method for the Riesz fractional differential eigenvalue problems.
Firstly, we expand $u_N(x)$ with the unknowns  $\{u_j\}^{N}_{j=0}$ as follows,
\begin{equation} \label{expression}
\displaystyle u_N(x)=\sum^{N}_{j=0}u_{j}\phi_j(x).
\end{equation}
By inserting the expansions of \eqref{expression} into \eqref{weakn} and taking the test functions $v_N=\phi_i(x), \, i=0,1,\dots,N $, then
problem \eqref{weakn}  can be written in the following matrix form:
\begin{equation} \label{matrixf}
S\overline{u}=\lambda_N B\overline{u},
\end{equation}
where
 the coefficient vector $\overline{u}=(u_0,u_1,\dots,u_N)^{\mathsf{T}}$, and
 the stiffness and mass matrices $S = (S_{ij})_{0\leq i,j\leq N}$ and $M=(M_{ij})_{0\leq i,j\leq N}$ are determinated by the following lemma.
 \begin{lemma}
 \label{entries}
 It holds that
\begin{align}
\label{entrys}
&S_{ij}=(D^{\alpha}\phi_j,D^{\alpha}\phi_i)=\delta_{ij} ,
\\
\label{entrym}
&M_{ij}=(\phi_j,\phi_i)= \begin{cases}\frac{(-1)^{\frac{j-i}{2}} \sqrt{\pi (2i+2\alpha+1)(2j+2\alpha+1)}\,  \Gamma(2\alpha+1) (i+j)! }{
2^{2\alpha+i+j+1}\, \Gamma(2\alpha+\frac{i+j}{2}+\frac{3}{2})\, \Gamma(\alpha+\frac{i-j}{2}+1)\,
\Gamma(\alpha+\frac{j-i}{2}+1) \left(\frac{i+j}{2}\right)! } , & i+j \text{ even},\\
0, &  i+j \text{  odd}.\end{cases}
\end{align}
\end{lemma}
\begin{proof}
Owing to  the orthogonality relation  \eqref{OrthJ}, one readily verifies \eqref{entrys}.

By a parity argument, one also  finds that  $M_{ij}=0$  whenever $i+j$ is odd.
To prove \eqref{entrym} for even $i+j$, we shall resort to the following
connection identity of two Jacobi polynomials \cite[Theorem 7.1.4]{Andrews99},
\begin{align}
\label{connect}
\begin{split}
P^{\gamma,\gamma}_n(x) =&\,
\sum_{\nu=0}^{\lfloor \frac{n}{2}\rfloor} c^{\gamma,\alpha}_{n,\nu} P^{\alpha,\alpha}_{n-2\nu}(x),
\\
c^{\gamma,\alpha}_{n,\nu}:=&\,\frac{(\gamma+1)_n(2\alpha+1)_{n-2\nu}(\gamma+\frac12)_{n-\nu} (\alpha+\frac32)_{n-2\nu} (\gamma-\alpha)_{\nu} }{(2\gamma+1)_n (\alpha+1)_{n-2\nu}(\alpha+\frac32)_{n-\nu} (\alpha+\frac12)_{n-2\nu} \nu!},
\end{split}
\end{align}
where the Pochhammer symbol $(a)_\nu=\frac{\Gamma(a+\nu)}{\Gamma(a)}$.

Indeed, by  Rodrigues' formula and integration by parts, one obtains
\begin{align*}
&\, \int_{-1}^1 (1-x^2)^{\alpha}P^{\alpha,\alpha}_j(x) (1-x^2)^{\alpha}P^{\alpha,\alpha}_i(x) dx
= \frac{(-1)^{j+i}}{2^{j+i}j!i!}
\int_{-1}^1 \partial_x^j(1-x^2)^{j+\alpha}  \cdot \partial_x^i(1-x^2)^{i+\alpha}  dx
\\
= &\,\frac{(-1)^{j}}{2^{j+i}j!i!}
\int_{-1}^1 (1-x^2)^{i+\alpha}   \partial_x^{j+i}(1-x^2)^{j+\alpha} dx
= \frac{(-1)^{i}  (j+i)!}{j!i!}
\int_{-1}^1 (1-x^2)^{2\alpha}  P_{j+i}^{\alpha-i,\alpha-i} (x)dx,
\end{align*}
where the  $P_{j+i}^{\alpha-i,\alpha-i}$ is a generalized Jacobi polynomial as discussed in \cite[\S4.22]{Szego39} whenever $\alpha-i\leq -1$.
To proceed, we utilize \eqref{connect} to  expand  $P^{\alpha-i,\alpha-i}_{j+i}$  as
$\sum_{\nu=0}^{\lfloor \frac{i+j}{2}\rfloor} c^{\alpha-i,2\alpha}_{i+j,\nu} P^{2\alpha,2\alpha}_{i+j-2\nu}$. Due  the orthogonality \eqref{OrthP} of Jacobi polynomials, we  derive  that
\begin{align*}
 &\int_{-1}^1 (1-x^2)^{\alpha}P^{\alpha,\alpha}_j(x) (1-x^2)^{\alpha}P^{\alpha,\alpha}_i(x) dx
 =  \frac{(-1)^{i}  (j+i)!}{j!i!}
\int_{-1}^1 (1-x^2)^{2\alpha}  c^{\alpha-i,2\alpha}_{j+i,\frac{j+i}{2}} P_{0}^{2\alpha,2\alpha} (x)dx
 \\
 =&\,\frac{(-1)^{i}  (j+i)!}{j!i!} \frac{(\alpha-i+1)_{j+i}  (\alpha-i+\frac12)_{\frac{j+i}{2}}  (-i-\alpha)_{\frac{j+i}{2}} }{(2\alpha-2i+1)_{j+i}  (2\alpha+\frac32)_{\frac{j+i}{2}} \frac{j+i}{2}!  }
\gamma^{2\alpha,2\alpha}_0 
\\
=&\,\frac{(-1)^{\frac{j-i}{2}}(j+i)!}{2^{j+i}j!i!} \frac{(-j-\alpha)_{\frac{j+i}{2}}    (-i-\alpha)_{\frac{j+i}{2}} }{ (2\alpha+\frac32)_{\frac{j+i}{2}} \frac{j+i}{2}!  }
\frac{\sqrt{\pi}\, \Gamma(2\alpha+1)}{ \Gamma(2\alpha+\frac32)}.
\end{align*}
Finally, \eqref{entrym} with even $i+j$ is  an immediate consequence of the above equation.
This ends the proof.
\end{proof}

It is worthy to point out that
the algebraic eigenvalue problem \eqref{matrixf} can be decoupled
into two
according to   even or odd modes.

\subsection{Numerical analysis}

In this subsection, we first give some estimates on the magnitudes of  the smallest and greatest numerical eigenvalues, and thus on the condition number of the mass matrix $M$.
Then the error estimates for both the eigenvalues and  eigenfunctions are presented.

Let $\{(\lambda_{i,N},\psi_{i,N})\}_{i=1}^{N+1}$ be the eigensolutions of \eqref{weakn} such that
$\lambda_{1,N} \le \lambda_{2,N}\le \cdots \le \lambda_{N+1,N}$.
The following Lemma indicates each the numerical eigenvalue satisfies  $\lambda_{i,N}\lesssim N^{4\alpha}$.

\begin{lemma}\label{inverse} (Inverse inequality) Suppose $ u_N \in \mathbb{F}_N^{-\alpha,-\alpha}(\Lambda)$, then
\begin{equation}
\label{INV}
\|D^{\alpha}u_N\|^2 \lesssim  N^{4\alpha}\, \|u_N\|^2 .
\end{equation}
\end{lemma}

\begin{proof}
For any
$
u_N(x)=\displaystyle \sum_{i=0}^N u_i\mathcal{J}_i^{-\alpha,-\alpha}(x) \in \mathbb{F}_N^{-\alpha,-\alpha}(\Lambda),
$
we have
\begin{align}
\label{abnd}
\|D^{\alpha}u\|^2
=&  \displaystyle \sum_{i=0}^{N} \frac{\Gamma(i+2\alpha+1)}{i!} u_i^2 \gamma_{i}^{\alpha,\alpha}
\lesssim  N^{2\alpha} \displaystyle \sum_{i=0}^{N} u_i^2 \gamma^{\alpha,\alpha}_n,
\end{align}
where  we have used the following estimate  \cite[(3.26)]{MCS16} for the inequality sign above,
\begin{equation*}
\displaystyle \frac{\Gamma(i+2\alpha+1)}{i!} \lesssim i^{2\alpha}\lesssim N^{2\alpha},\qquad 0\le i\le N.
\end{equation*}

Meanwhile, we note that  \cite[Corollary 6.2 of page 99]{BDM07},
$$  \|\psi_N\|_{\omega^{\alpha,\alpha}} \lesssim N^{\alpha}\|\psi_N\|_{\omega^{2\alpha,2\alpha}},
\qquad \forall \psi_N(x) \in \mathbb{P}_N(\Lambda).$$
Since $\omega^{-\alpha,-\alpha} u_N\in \mathbb{P}_N(\Lambda)$, we now derive that
\begin{align}
\|u_N\|^2= & \big\|\omega^{-\alpha,-\alpha} u_N\big\|_{\omega^{2\alpha,2\alpha} }^2
 \gtrsim  N^{-2\alpha} \big\|\omega^{-\alpha,-\alpha} u_N\big\|_{\omega^{\alpha,\alpha}}^2
= N^{-2\alpha}  \sum_{i=0}^{N} u_i^2 \gamma_i^{\alpha,\alpha}  .
\label{bbnd}
\end{align}
A combination of \eqref{abnd} and \eqref{bbnd} then yields \eqref{INV}. This completes the proof.
\end{proof}

The following theorem gives some sharp estimates on the numerical eigenvalues.
\begin{theorem}\label{COND}  Its holds that
\begin{align}
\label{lambdaN}
&\lambda_{1,N}= \mathcal{O}(1), \quad  \lambda_{N+1,N}= \mathcal{O}(N^{4\alpha}),
\end{align}
as $N$ tends to infinity. Further,
the spectral condition number of the mass matrix $M$ satisfies
\begin{align}
\label{condM}
&\chi_N(M)=\frac{\lambda_{N+1,N}}{\lambda_{1,N}}=\mathcal{O}( N^{4\alpha}).
\end{align}
\end{theorem}

\begin{proof}  Since \eqref{condM} is an immediate consequence of \eqref{lambdaN},
we only need to prove \eqref{lambdaN}.

Thanks to the Min-Max principle and Lemma \ref{entries},
\begin{align*}
 \lambda_1=&\min_{u\in H^s_*(I)}  \frac{\|D^{\alpha} u\|^2}{ \|u\|^2} \le
 \min_{u\in \mathbb{F}^{-\alpha,-\alpha}_N(\Lambda)}  \frac{\|D^{\alpha} u\|^2}{ \|u\|^2}= \lambda_{1,N},
 \\
 \lambda_{1,N} =&\,\min_{u\in \mathbb{F}^{-\alpha,-\alpha}_N(\Lambda)}  \frac{\|D^{\alpha} u\|^2}{ \|u\|^2}
 \le \frac{\|D^{\alpha} \phi_0\|^2}{ \|\phi_0\|^2} = \frac{1}{M_{00}}
 \\ =&\, \frac{2^{2\alpha+1}\Gamma(2\alpha+\frac32)\Gamma(\alpha+1)^2}{\sqrt{\pi}\, \Gamma(2\alpha+2)} = \frac{\Gamma(2\alpha+\frac32)\Gamma(\alpha+1)}{\Gamma(2\alpha+2)\Gamma(\alpha+\frac32)},
\end{align*}
where $M_{00}$
which yields that $\lambda_{1,N}=\mathcal{O}(1)$.

Next, by the inverse equality \eqref{INV}, we find that
\begin{align*}
\lambda_{N+1,N} = \max_{\mathbb{F}^{-\alpha,-\alpha}_N(\Lambda)} \frac{\|D^{\alpha} u\|^2}{\|u\|^2 } \lesssim N^{4\alpha}.
\end{align*}
To prove $\lambda_{N+1,N} = \mathcal{O}( N^{4\alpha}) $, it then suffices to verify that
\begin{align}
\label{etaN}
&\|D^{\alpha}\eta_{N}\|^2 = \mathcal{O}(N^{8\alpha+2}), \quad
\|\eta_{N}\|^2 = \mathcal{O}(N^{4\alpha+2}), & N \to +\infty.
\end{align}
where
\begin{align*}
\eta_N(x) = \frac{\Gamma(N+4\alpha+3)} {\Gamma(N+2\alpha+2)}  (1-x^2)^{\alpha}P^{2\alpha+1,2\alpha+1}_N(x)\in\mathbb{F}^{-\alpha,-\alpha}_N(\Lambda).
\end{align*}
Indeed, by \eqref{connect},
\begin{align*}
\eta_N(x)& = (1-x^2)^{\alpha}
\sum_{\nu=0}^{\lfloor N/2 \rfloor} \frac{4\, \Gamma(N-2\nu+4\alpha+1)\, (N-2\nu+2\alpha+1/2) }
{ \Gamma(N-2\nu+2\alpha+1) }
P^{2\alpha,2\alpha}_{N-2\nu}(x).
\end{align*}
Then the orthogonality \eqref{OrthP} reveals that
\begin{align*}
\|\eta_N\|^2 &=  \sum_{\nu=0}^{\lfloor N/2 \rfloor}
\frac{16\, \Gamma(N-2\nu+4\alpha+1)^2\, (N-2\nu+2\alpha+1/2)^2 \times  2^{4\alpha+1}\, \Gamma(N-2\nu+2\alpha+1)^2  }
{ \Gamma(N-2\nu+2\alpha+1)^2   \times  (2N-4\nu+4\alpha+1)\,(N-2\nu)!\, \Gamma(N-2\nu+4\alpha+1) }
\\
&=  2^{4\alpha+4}  \sum_{\nu=0}^{\lfloor N/2 \rfloor}\frac{\Gamma(N-2\nu+4\alpha+1)  (N-2\nu+2\alpha+1/2) }
{  (N-2\nu)! }
\\
&=  \mathcal{O}(N^{4\alpha+2}),
\end{align*}
where the last equality sign is derived by using  the fact  that \cite[(1.66)]{SKM93}
$$ \lim_{z\to +\infty}\frac{\Gamma(z+\alpha)}{\Gamma(z)z^{\alpha}} = 1, \qquad \alpha>0.  $$
Meanwhile,  the connection relation \eqref{connect} gives
\begin{align*}
\eta_N(x)& =  (1-x^2)^{\alpha}   \frac{ \Gamma(4\alpha+3)\Gamma(\alpha+1/2)}
{\Gamma(2\alpha+2) \Gamma(2\alpha+1) \Gamma(2\alpha+3/2)}
\\
&\times \sum_{\nu=0}^{\lfloor N/2 \rfloor}
\frac{(N-2\nu+\alpha+1/2)
 \Gamma(N-2\nu+2\alpha+1) \Gamma(N-\nu+2\alpha+3/2) \Gamma(\nu+\alpha+1)}
{ \Gamma(N-2\nu+\alpha+1) \Gamma(N-\nu+\alpha+3/2) \nu!}
P^{\alpha,\alpha}_{N-2\nu}(x),
\end{align*}
which together with the orthogonality \eqref{OrthJ} impiles
\begin{align*}
\|D^{\alpha}\eta_N\|^2 &=  \frac{2^{2\alpha} \Gamma(4\alpha+3)^2\Gamma(\alpha+1/2)^2   }
{\Gamma(2\alpha+2)^2 \Gamma(2\alpha+1)^2 \Gamma(2\alpha+3/2)^2}
\\
&\times \sum_{\nu=0}^{\lfloor N/2 \rfloor}
\frac{(N-2\nu+\alpha+1/2)\,
\Gamma(N-2\nu+2\alpha+1)^2\, \Gamma(N-\nu+2\alpha+3/2)^2\, \Gamma(\nu+\alpha+1)^2 }
{ (N-2\nu)!^2\,\Gamma(N-\nu+\alpha+3/2)^2\, \nu!^2 }
\\
&=\mathcal{O}(N^{8\alpha+2}).
\end{align*}
This gives \eqref{etaN} and the proof is now completed.
\end{proof}

We now concentrate on the error estimates for the numerical eigenvalues and eigenfunctions.
Let us  define the orthogonal projector 
$\Pi_N^{-\alpha,-\alpha}: H^{\alpha}_*(\Lambda) \rightarrow \mathbb{F}_N^{-\alpha,-\alpha}(\Lambda)$,
    such that for all $u \in H^{\alpha}_*(\Lambda)$,
\begin{equation*}
\begin{array}{l}
(D^{\alpha}(\Pi_N^{-\alpha,-\alpha}u-u),D^{\alpha}v_N)=0,\qquad  v_N \in  \mathbb{F}_N^{-\alpha,-\alpha}(\Lambda).
\end{array}
\end{equation*}
Then we introduce  the Sobolev space for any $\alpha>0$, $\nu\in \{ 0,1,\dots,\lfloor\alpha\rfloor\}$,   $m\in \NN_0, m\ge 2\nu-\lfloor\alpha\rfloor$,
\begin{align*}
\mathcal{B}_{\alpha,\nu}^{m}(\Lambda) =
\{v\in L^2_{\omega^{-\alpha,-\alpha}}(\Lambda) :\frac{d^l}{dx^l} D^{2\alpha-2\nu}v  \in L^2_{\omega^{\alpha-2\nu +l,\alpha-2\nu+l}},\  0\le l\le m  \}.
\end{align*}

\begin{lemma}\label{perror}
Assume   $\alpha>0$, $\nu\in \{ 0,1,\dots,\lfloor\alpha\rfloor\}$. Then for all $u \in \mathcal{B}_{\alpha,\nu}^{m}(\Lambda)$ with $m\ge 2\nu-\lfloor\alpha\rfloor$,  we have the following estimate,
\begin{align}
\label{Erra}
&\|D^{\alpha}(\Pi_N^{-\alpha,-\alpha} u-u)\|\lesssim N^{2\nu-\alpha-m} \Big\|\frac{d^m}{dx^m} D^{2\alpha-2\nu}u\Big\|_{\omega^{\alpha-2\nu+m,\alpha-2\nu+m}}.
\end{align}
\end{lemma}
\begin{proof}
By \eqref{derpol}, one has
\begin{align*}
\frac{d^m}{dx^m} &D^{2\alpha-2\nu} \mathcal{J}^{-\alpha,-\alpha}_i(x)
=(-1)^{\lfloor \alpha+\frac12\rfloor} \frac{\Gamma(i+2\alpha-2\nu+1)}{2^{-2\nu} i!} \frac{d^m}{dx^m} P^{\alpha-2\nu ,\alpha-2\nu}_{i+2\nu}(x)
\\
&= (-1)^{\lfloor \alpha+\frac12\rfloor} \frac{\Gamma(i+2\alpha-2\nu+m+1)}{2^{m-2\nu} i!} P^{\alpha-2\nu+m ,\alpha-2\nu+m}_{i+2\nu-m}(x).
\end{align*}
Further by the orthogonality of the Jacobi polynomials \eqref{OrthP}, one obtains
\begin{align*}
\Big(\frac{d^m}{dx^m}& D^{2\alpha-2\nu} \mathcal{J}^{-\alpha,-\alpha}_i, \frac{d^m}{dx^m} D^{2\alpha-2\nu} \mathcal{J}^{-\alpha,-\alpha}_n\Big)_{\omega^{\alpha-2\nu+m,\alpha-2\nu+m}}
\\
&= \frac{2^{2\alpha+1} \Gamma(i+\alpha+1)^2 }{(2i+2\alpha+1)i!^2  }
\frac{\Gamma(i+2\alpha-2\nu+m+1)}{\Gamma(i+2\nu-m+1)} \delta_{in}.
\end{align*}
Thus for any $u=\sum_{i=0}^{\infty}u_i\mathcal{J}^{-\alpha,-\alpha}_i\in \mathcal{B}^{m}_{\alpha,\nu}(\Lambda)\subseteq  H_*^{\alpha}(I)
$,
\begin{align*}
\Big\|\frac{d^m}{dx^m} D^{2\alpha-2\nu}u\Big\|^2_{\omega^{\alpha-2\nu+m,\alpha-2\nu+m}}
=\sum_{i=0}^{\infty} \frac{2^{2\alpha+1} \Gamma(i+\alpha+1)^2 }{(2i+2\alpha+1)i!^2  }
\frac{\Gamma(i+2\alpha-2\nu+m+1)}{\Gamma(i+2\nu-m+1)}  |u_i|^2.
\end{align*}
Meanwhile, by the definition of $\Pi_N^{-\alpha,-\alpha}$,
it can be readily checked that
\begin{align*}
\Pi_N^{-\alpha,-\alpha} u = \sum_{i=0}^N u_i\mathcal{J}^{-\alpha,-\alpha}_i.
\end{align*}
Then one  gets from \eqref{OrthJ} that
\begin{align*}
\|D^{\alpha}(&\Pi_N^{-\alpha,-\alpha} u-u)\|^2
= \sum_{i=N+1}^{\infty} \frac{2^{2\alpha+1} \Gamma(i+\alpha+1)^2 }{(2i+2\alpha+1)i!^2  } |u_i|^2
\\
&= \frac{\Gamma(N+2\nu-m+1)} {\Gamma(N+2\alpha-2\nu+m+1)}
\\
&\quad \times \sum_{i=N+1}^{\infty} \frac{(N+2\nu-m+1)_{i-N}} {\Gamma(N+2\alpha-2\nu+m+1)_{i-N}}\frac{2^{2\alpha+1} \Gamma(i+\alpha+1)^2 }{(2i+2\alpha+1)i!^2  }
\frac{\Gamma(i+2\alpha-2\nu+m+1)}{\Gamma(i+2\nu-m+1)}  |u_i|^2
\\
&\le \frac{\Gamma(N+2\nu-m+2)} {\Gamma(N+2\alpha-2\nu+m+2)}
 \sum_{i=0}^{\infty} \frac{2^{2\alpha+1} \Gamma(i+\alpha+1)^2 }{(2i+2\alpha+1)i!^2  }
\frac{\Gamma(i+2\alpha-2\nu+m+1)}{\Gamma(i+2\nu-m+1)}  |u_i|^2
\\
&\lesssim N^{4\nu-2m-2\alpha}
\Big\|\frac{d^m}{dx^m} D^{2\alpha-2\nu}u\Big\|^2_{\omega^{\alpha-2\nu+m,\alpha-2\nu+m}},
\end{align*}
where the following estimate has been used for the first inequality sign,
\begin{align*}
\frac{(N+2\nu-m+1)_{i-N}} {\Gamma(N+2\alpha-2\nu+m+1)_{i-N}} \le 1, \qquad \alpha+m-2\nu \ge 0.
\end{align*}
This finally completes the proof.
\end{proof}

Recall that
$a(\cdot,\cdot)$ is symmetric, continuous
and coercive on  $H^{\alpha}_*(\Lambda)\times H^{\alpha}_*(\Lambda)$,
 $b(\cdot,\cdot)$  is continuous on
$L^2(\Lambda)\times L^2(\Lambda)$,  and $H^{\alpha}_*(\Lambda)$
is compactly imbedded in $L^2(\Lambda)$.
Thus,
by the approximation theory of  Babu\v{s}ka and Osborn on
the Ritz method for self-adjoint and positive-definite  eigenvalue problems
\cite[pp. 697-700]{BO91},
we now arrive at the following main theorem.
\begin{theorem}
Let $\lambda_j$ be an eigenvalue of \eqref{eigp1} with the geometric multiplicity $q$  and assume that
$\lambda_j=\lambda_{j+1}=\dots=\lambda_{j+q-1}$. Denote
   \begin{align*}
  E(\lambda_j) := \left\{ \psi \text{ is an eigenfunction corresponding to } \lambda_j    \text{ with } \|D^{\alpha}\psi\|=1     \right\} .
\end{align*}
Further suppose $\psi\in \mathcal{B}_{\alpha,\nu}^m(\Lambda)$ with $\nu\in \{0,1,\dots,\lfloor \alpha\rfloor\}$ and  $m\ge 2\nu-\alpha$ for any $\psi\in E(\lambda_j)$.
  It  holds that
   \begin{align*}
           & 0 \le \lambda_{i,N}-\lambda_j \lesssim N^{4\nu-2\alpha-2m} \sup_{\psi\in E(\lambda_j)} \Big\|\frac{d^m}{dx^m} D^{2\alpha-2\nu}\psi\Big\|_{\omega^{\alpha-2\nu+m,\alpha-2\nu+m}}^2, \quad  i=j,j+1,\dots,j+q-1.
   \end{align*}

    Let  $\psi_{i,N}$ be an  eigenfunction  corresponding to $\lambda_{i,N}$
    such that $\|D^{\alpha}\psi_{i,N}\|=1$.  Then for $i=j,j+1,\dots,j+q-1$, there holds that
  \begin{align*}
           \inf_{u\in E(\lambda_j)} \|u-\psi_{i,N}\|_{\alpha,*} \lesssim  N^{2\nu-\alpha-m}  \sup_{\psi\in E(\lambda_j)}  \Big\|\frac{d^m}{dx^m} D^{2\alpha-2\nu}\psi\Big\|_{\omega^{\alpha-2\nu+m,\alpha-2\nu+m}}
.
   \end{align*}
    Moreover, for  $\psi\in E(\lambda_j)$, there exists a function  $v_N\in \mathrm{span}\{ \psi_{j,N},\dots,\psi_{j+q-1,N} \} $ such that
  \begin{align*}
         \|\psi-v_{N}\|_{\alpha,*} \lesssim   N^{2\nu-\alpha-m}  \sup_{\psi\in E(\lambda_j)}  \Big\|\frac{d^m}{dx^m} D^{2\alpha-2\nu}\psi\Big\|_{\omega^{\alpha-2\nu+m,\alpha-2\nu+m}}
.
   \end{align*}

\end{theorem}

\section{Numerical results}
\label{sec:results}
In this section, we will give some numerical results to illustrate  the accuracy and
efficiency of the  Jacobi-Galerkin spectral method for the Riesz fractional differential
eigenvalue problems on the reference domain and validate other  theoretical results related.

\begin{example}
 Fractional derivative order $2\alpha\in(0,2]$.
\end{example}
We   begin with $2\alpha\in(0,2]$.  This example  includes
three parts. The first test is to verify the spectral accuracy, while another two tests are to confirm the
Weyl-type asymptotic law  and the asymptotic  behavior of the linear algebraic system's condition number respectively.
\subsection{Accuracy and convergence test -- algebraic order or exponential order?}

The  five leading eigenvalues obtained by our Jacobi-Galerkin spectral method with $N=64$,  are  listed  in Table \ref{tab} for different $2\alpha$.  They are at least 10 digits accurate,
with significant places being  estimated from  the results  with $N=48$. This demonstrates the
accuracy of  our proposed  method. In Table \ref{eig3}, we report the   first three eigenvalues by Jacobi-Galerkin spectral method with $N=64$ for other various  $2\alpha$ and   the numerical approximations to $\lambda_n$ provided by the low order method in Ref. \cite{ZRK07} with  the total degrees of freedom $5000$. We can observe that they are in good  agreement.
\begin{table}[htp]
\centering
\caption{The 5 leading eigenvalues for each $2\alpha$,
computed with $N=64$, listed with significant places estimated from the $N=48$ results.} \vspace{2mm}
\footnotesize
\begin{tabular}{ |l |l|l|l|l|} \hline
$2\alpha=1.2$&$2\alpha=1.4$&$2\alpha=1.6$&$2\alpha=1.8$&$2\alpha=2.0$ \\ \hline
1.29699577674  &1.48323343195   &1.7282959570964  &2.048734983129 &2.467401100272 \\ \hline
3.4867305364   &4.45817398389    &5.75634828003  &7.50311692608 &9.86960440108 \\ \hline
5.911679975    &8.1507167266    &11.31189330097  &15.799894163321 & 22.2066099024  \\ \hline
8.534441423    &12.424353637    &18.1773428791  &26.724243284906 &39.47841760435 \\ \hline
11.29243001    &17.162347657    &26.1872040516  &40.11423380506  &61.68502750680  \\ \hline
\end{tabular}
\label{tab}
\end{table}

\begin{table}[htp]
\centering
\caption{Comparison of the first three eigenvalues by Jacobi-Galerkin spectral method with $N=64$ and numerical approximations to $\lambda_n$ obtained by the method of
Ref.\cite{ZRK07} with $5000\times 5000$ matrices.} \vspace{2mm}
\begin{tabular}{ |l |l|l|l|l|r|r|} \hline
 & \multicolumn{2}{|c|}{$\lambda_1$}& \multicolumn{2}{|c|}{$\lambda_2$}& \multicolumn{2}{|c|}{$\lambda_3$} \\ \hline
$2\alpha$ &Present&Ref.\cite{ZRK07}&Present&Ref.\cite{ZRK07}&Present&Ref.\cite{ZRK07} \\ \hline
0.01 &0.9966&0.997&  1.0087&1.009&  1.0137& 1.014\\
0.1  &0.9725&0.973&  1.0921&1.092&  1.1473& 1.148\\
0.2  &0.9574&0.957&  1.1965&1.197&  1.3190& 1.320\\
0.5  &0.9701&0.970&  1.6015&1.601&  2.0288& 2.031 \\
1.0  &1.1577&1.158&  2.7547&2.754&  4.3168& 4.320 \\
1.5  &1.5975&1.597&  5.0597&5.059&  9.5943& 9.957\\
1.8  &2.0487&2.048&  7.5031&7.501& 15.7998&15.801\\
1.9  &2.2440&2.243&  8.5957&8.593& 18.7168&18.718\\
1.99 &2.4436&2.2442& 9.7331&9.729& 21.8286&21.829\\ \hline
\end{tabular}
\label{eig3}
\end{table}

To investigate the convergence order of our method,  we utilize MAPLE  with a precision of 30 digits
to establish  an even  more accurate computational environment.
Taking the eigenvalues computed with $N=200$ as the reference  eigenvalues, the errors of the first eigenvalue versus polynomial degree $N$ have been plotted in Figs.\,\ref{error2} and \ref{error1} with various $2\alpha$ in logarithm-logarithm scale (left) and semi-logarithm (center). The down-bending curves in the logarithm-logarithm plots  indicate that
the errors decay faster than algebraically as  $N$ increases, and  
asymptotically exponential orders of convergence   for fractional differential order $2\alpha$ can   be observed  in the semi-logarithm plot for large $N.$ Indeed, the log-log plot 
of $\log(\lambda_{1,N}-\lambda_1)$ (right) reveals that our Jacobi-Galerkin spectral method 
converges at a sub-geometric rate $\mathcal{O}(\exp(-qN^r))$ with $0<r<1$ for a given $\alpha\not \in \NN_0$.

Clearly, the first eigenvalue grows exponentially as   $2\alpha$ increases as displayed in the left side of  Fig.\,\ref{First}, and it   converges algebraically to
the first Laplacian eigenvalue  as  $2\alpha$ approaches $2$ (see the right side of  Fig.\,\ref{First}).
These also confirm the asymptotically exponential order of convergence of our method.
By the way, the dashed line for the Gamma function  $\Gamma(2\alpha+1)$  in the left side of Fig.\,\ref{First} shows the lower bound of the numerical eigenvalues,
which is in agreement to with the Poincar\'{e} inequality in  Proposition \ref{poincare}.

\begin{figure}[h!bt]
\begin{center}
\includegraphics[width=0.32\textwidth,angle=0]{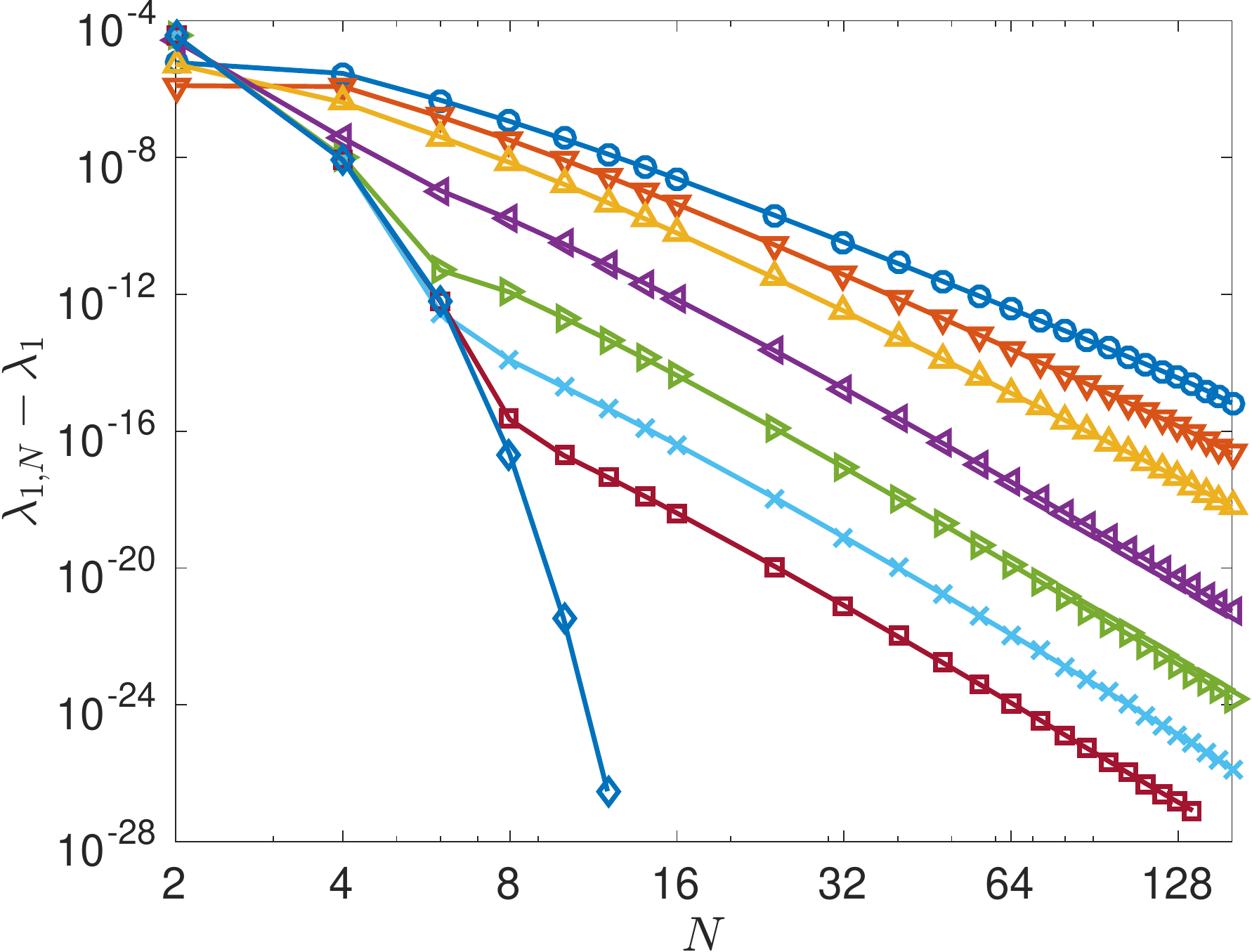}
\includegraphics[width=0.32\textwidth,angle=0]{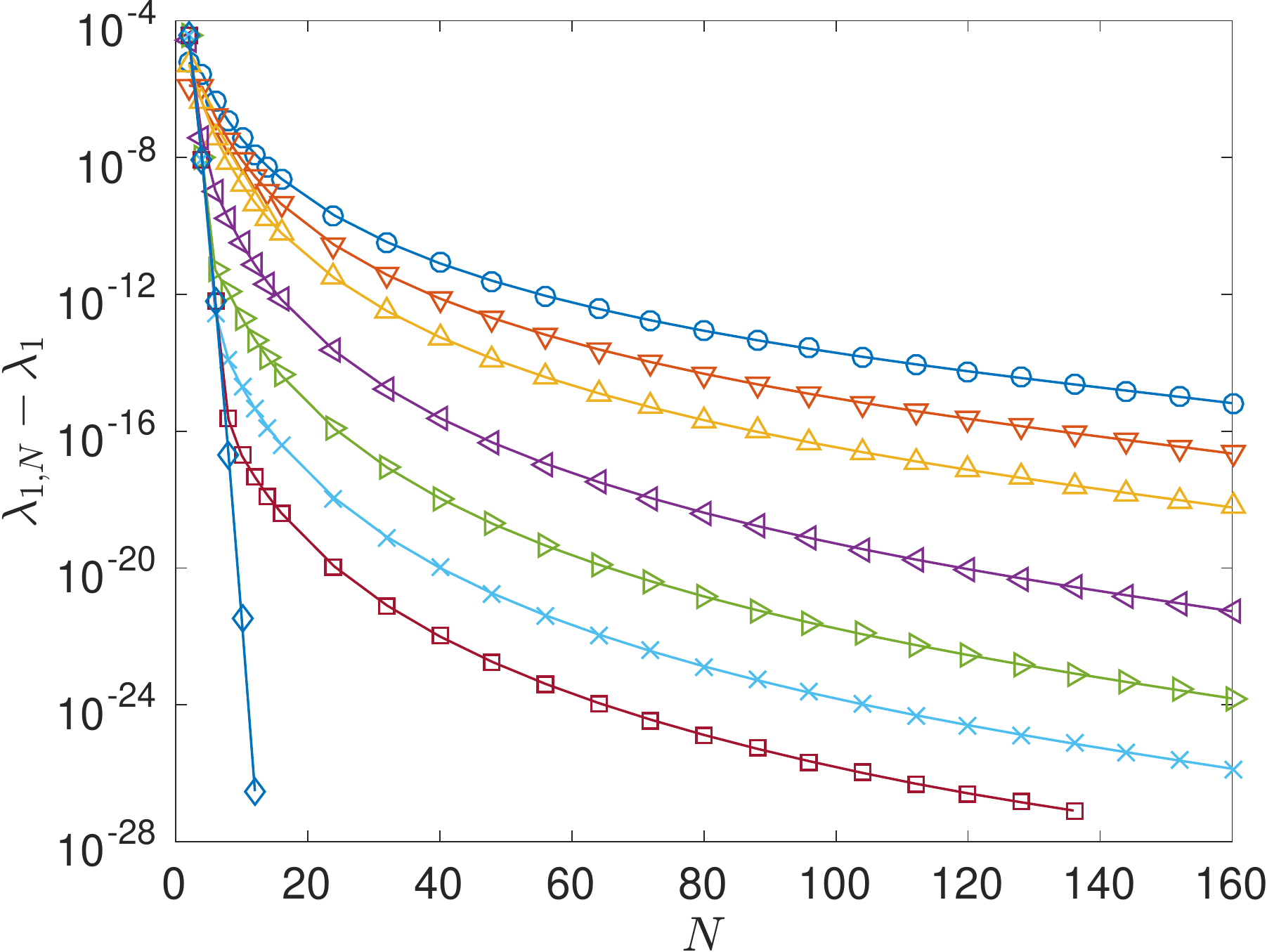}
\includegraphics[width=0.3\textwidth,angle=0]{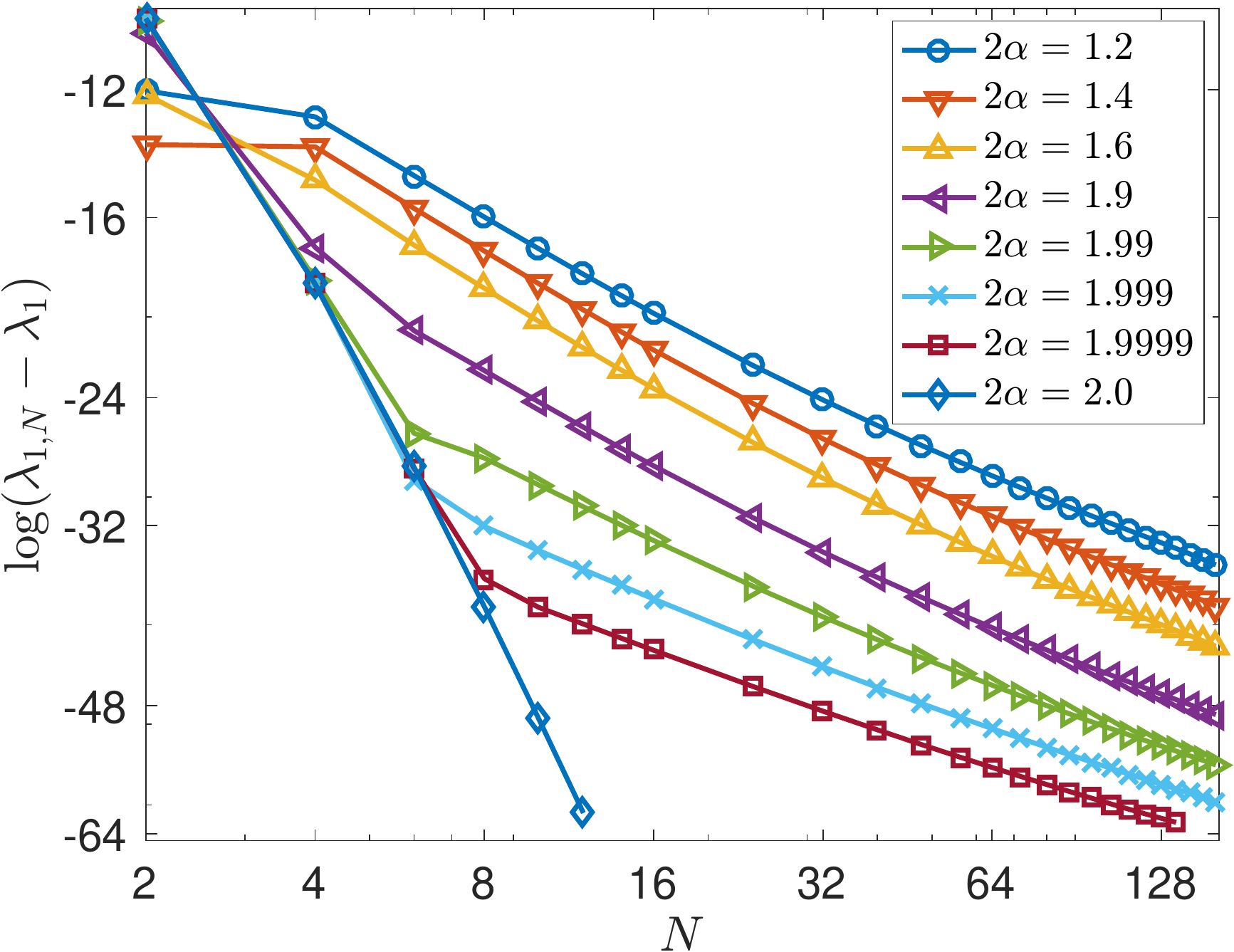}
\end{center}
\caption{Errors of the first eigenvalue versus polynomial degree $N$  with various fractional orders. Left: $\lambda_{1,N}-\lambda_1$ in logarithm-logarithm scale; Center:  $\lambda_{1,N}-\lambda_1$  in semi-logarithm scale; Right: 
$\log(\lambda_{1,N}-\lambda_1)$ in logarithm-logarithm scale.}
\label{error2}
\end{figure}

\begin{figure}[h!bt]
\begin{center}
\includegraphics[width=0.32\textwidth,angle=0]{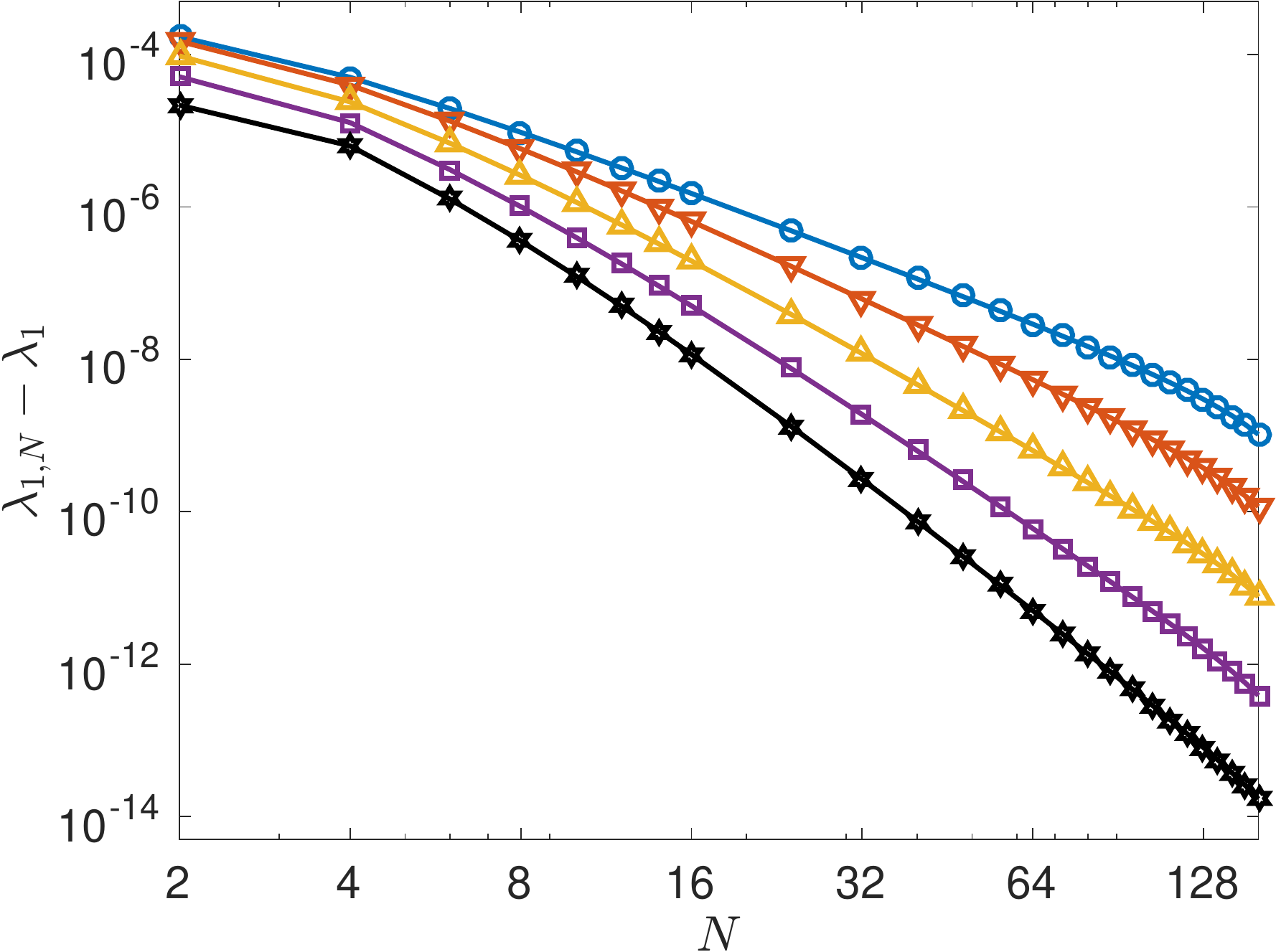}
\includegraphics[width=0.32\textwidth,angle=0]{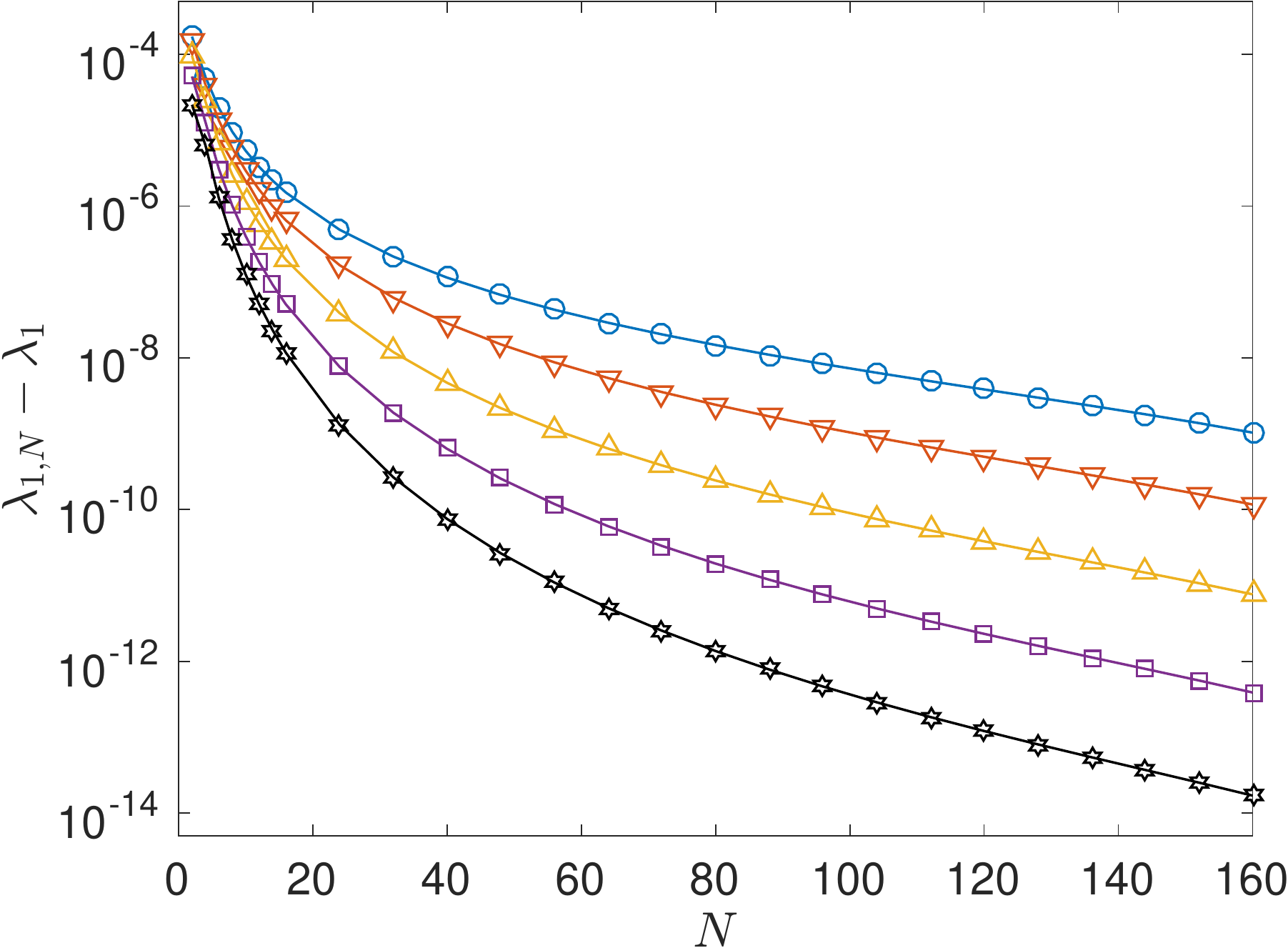}
\includegraphics[width=0.3\textwidth,angle=0]{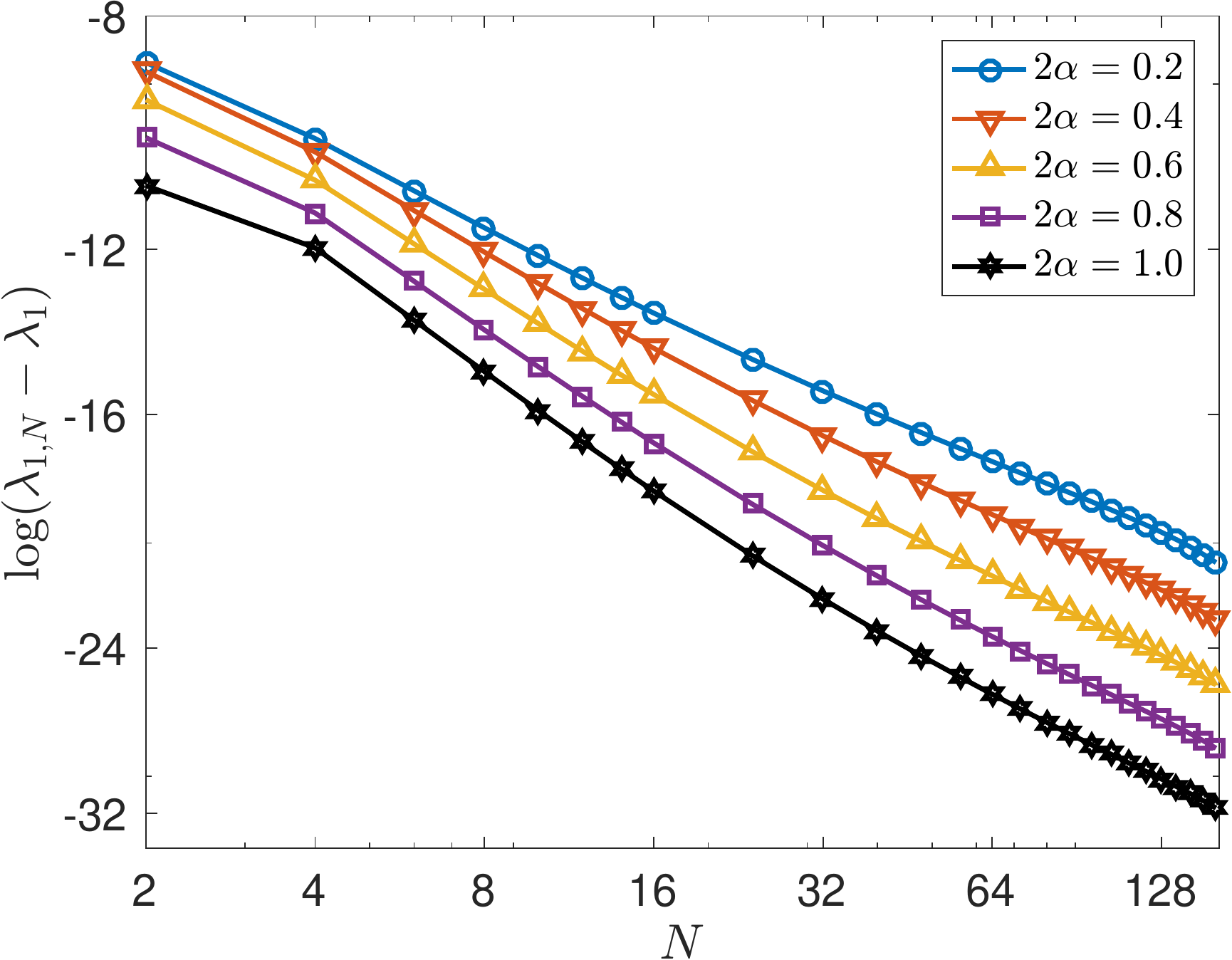}
\end{center}
\caption{Errors of the first eigenvalue versus polynomial degree $N$  with various fractional orders. Left: $\lambda_{1,N}-\lambda_1$ in logarithm-logarithm scale; Center:  $\lambda_{1,N}-\lambda_1$  in semi-logarithm scale; Right: 
$\log(\lambda_{1,N}-\lambda_1)$ in logarithm-logarithm scale.}

\label{error1}
\end{figure}

\vspace*{1em}

\begin{figure}[h!bt]
\hfill\includegraphics[width=0.35\textwidth,angle=0]{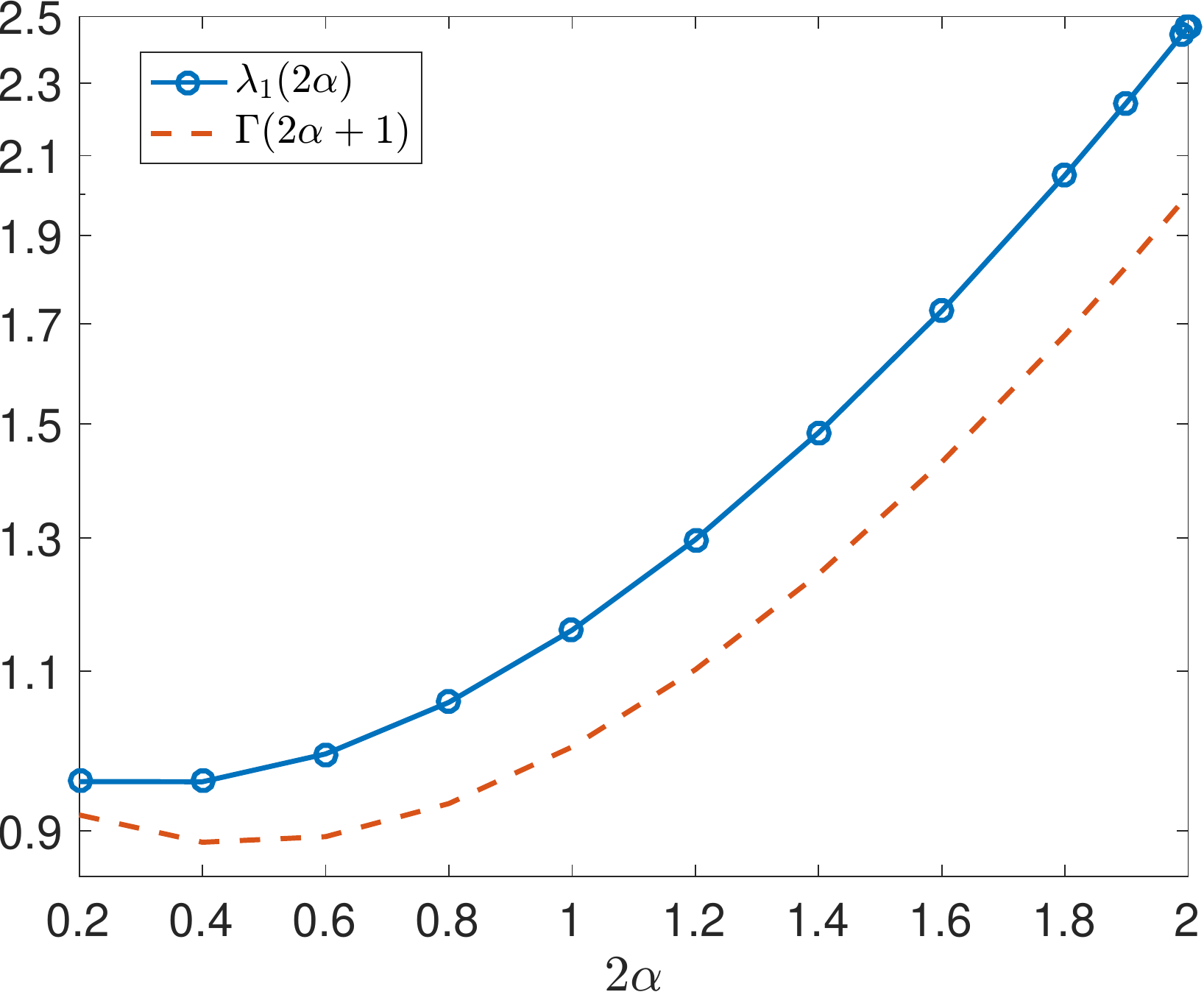}\hfill%
\includegraphics[width=0.35\textwidth,angle=0]{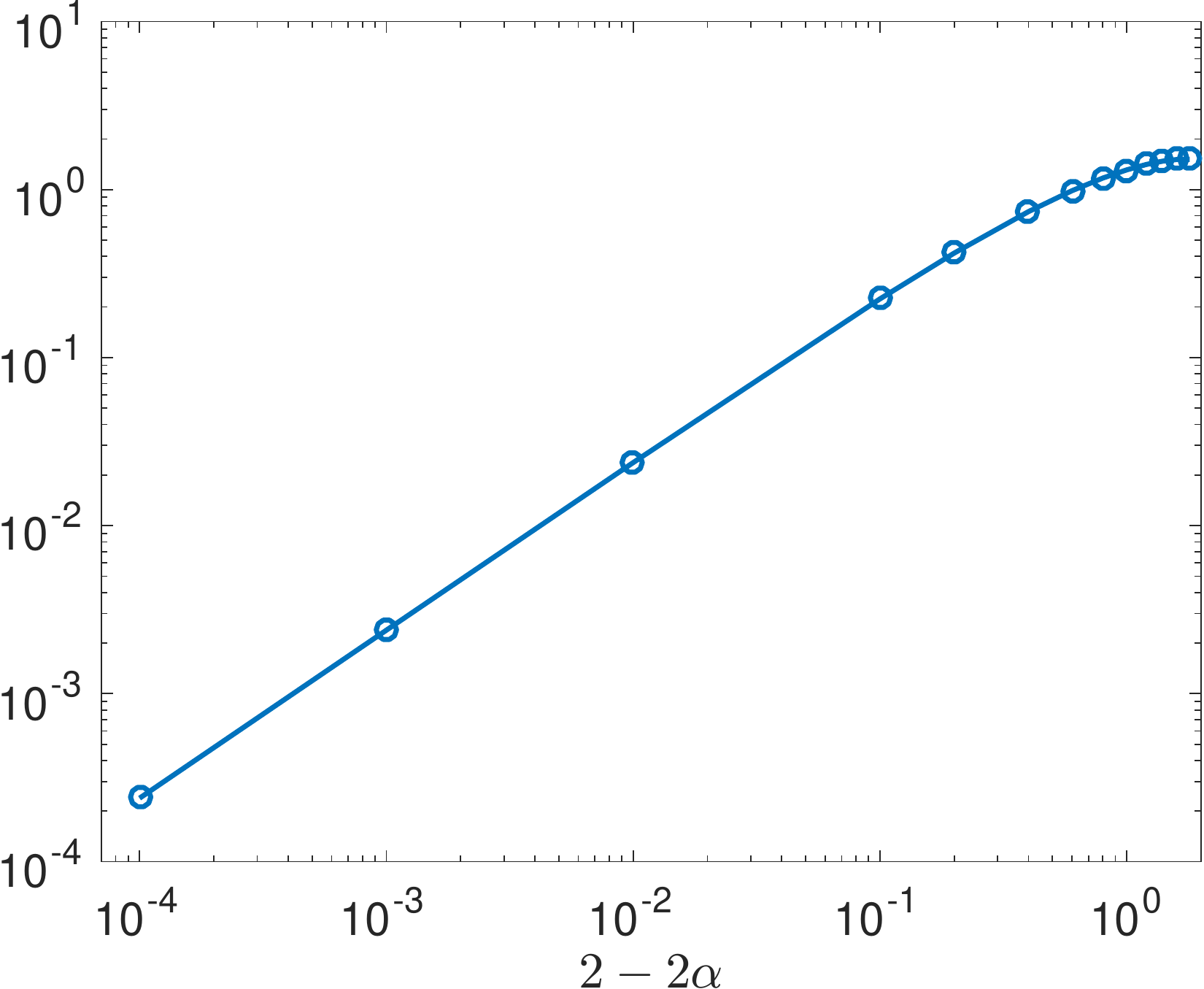}%
\hspace*{\fill}
\caption{Left: The first eigenvalue and the reference function $\Gamma(2\alpha+1)$ for different fractional partial order $2\alpha$ from $0.2$ to $2.0$ in semi-logarithm scale.
Right: $\lambda_1(2)-\lambda_1(2\alpha)$ versus $2-2\alpha$ in logarithm-logarithm scale.}
\label{eigmin}
\label{First}
\end{figure}

\begin{figure}[h!bt]
\hspace*{\fill}%
\includegraphics[width=0.335\textwidth,angle=0]{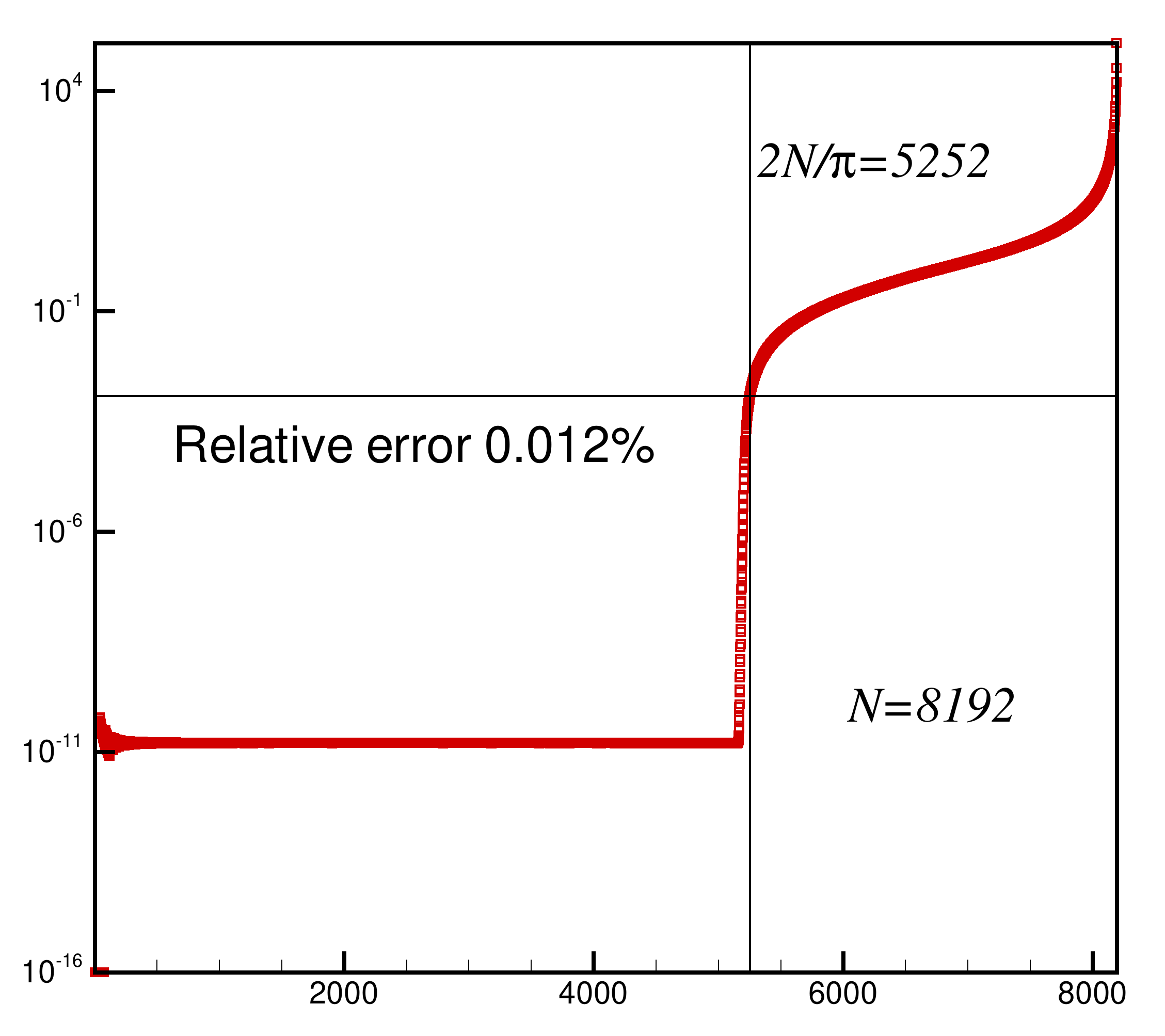}\hspace*{\fill}%
\includegraphics[width=0.38\textwidth,angle=0]{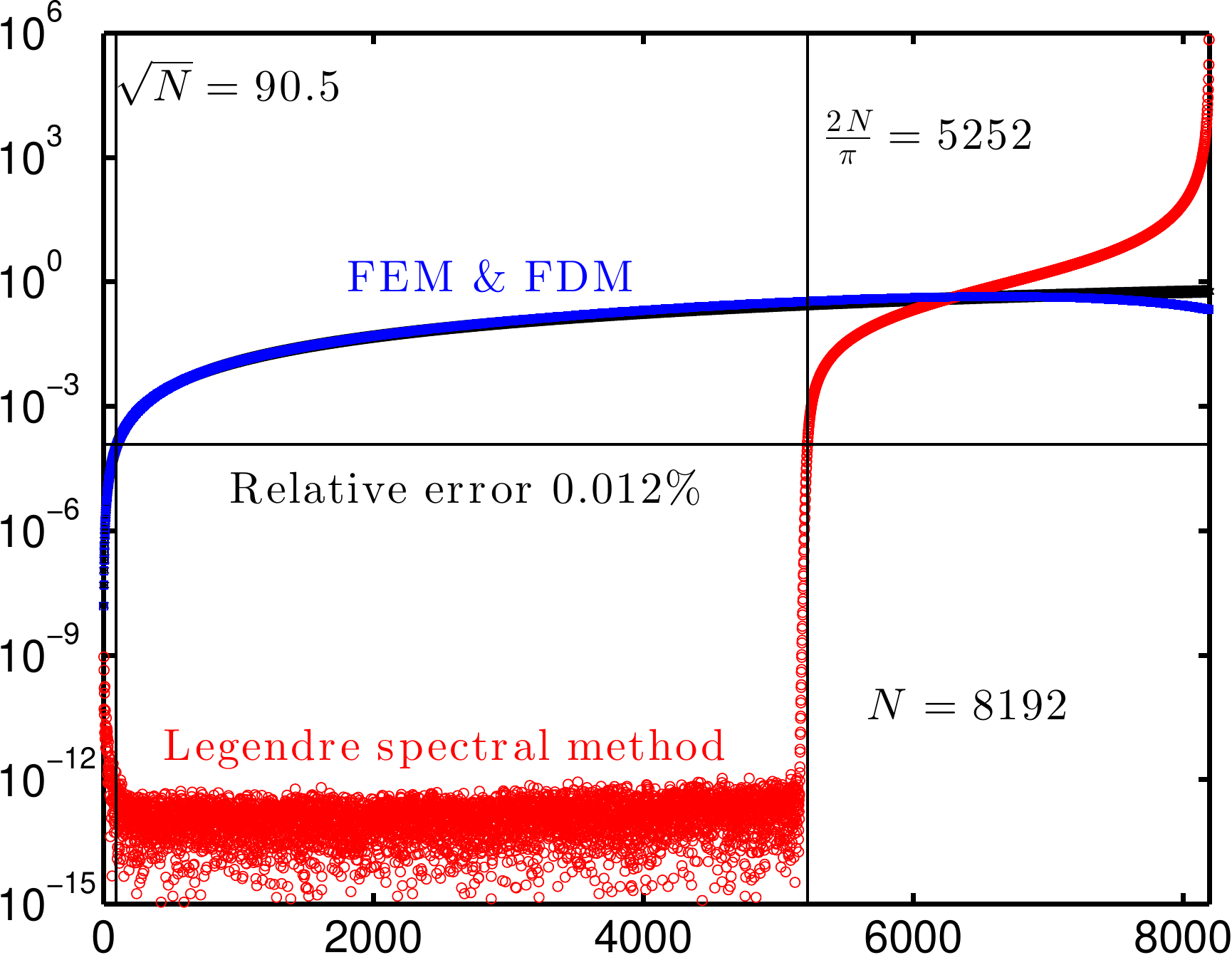}%
\hspace*{\fill}
\caption{Reliable eigenvalues and their relative errors. Left: $2\alpha=1.6$; Right: $2\alpha=2$.}
\label{erroreig}
\end{figure}
We plot  in Fig.\,\ref{erroreig} (left)  the relative errors for $8192$ numerical eigenvalues
for $2\alpha=1.6$ by applying the Jacobi-Galerkin spectral method with the polynomial degree  $N=8192$, where the errors have been evaluated by comparing these numerical eigenvalues with  the reference solution by $N=10240$.
The error curve of the spectral method cuts the intersection
of the horizonal line $y = 0.012 \%$ and the vertical line $x=2N/\pi $.  This observation confirms that there are about $2N/\pi$ numerical eigenvalues are reliable, for which the relative error converges at
a rate  of at least $\mathcal{O}(N^{-1})$. To make a comparison with the Legendre spectral method
for Laplacian eigenvalues, we excerpt a similar plot from \cite{Zhang15} on the right side, which is actually   provided by the third author of this paper. It indicates that the Jacobi-Galerkin spectral method
for  Riesz fractional eigenvalues behaves asymptotically  as the same as the  Legendre spectral for Laplacian eigenvalues, which strongly supports our hypothesis on the asymptotically exponential order
of convergence.

Finally, we plot in Fig.\,\ref{vector}  the first three eigenvectors with different fractional
 order $2\alpha=1.2, 1.6, 2.0$ and $N=32$. Here the eigenvectors are all normalized by $L^2$-norm.
\begin{figure}[h!bt]
\begin{center}
\includegraphics[width=0.32\textwidth,angle=0]{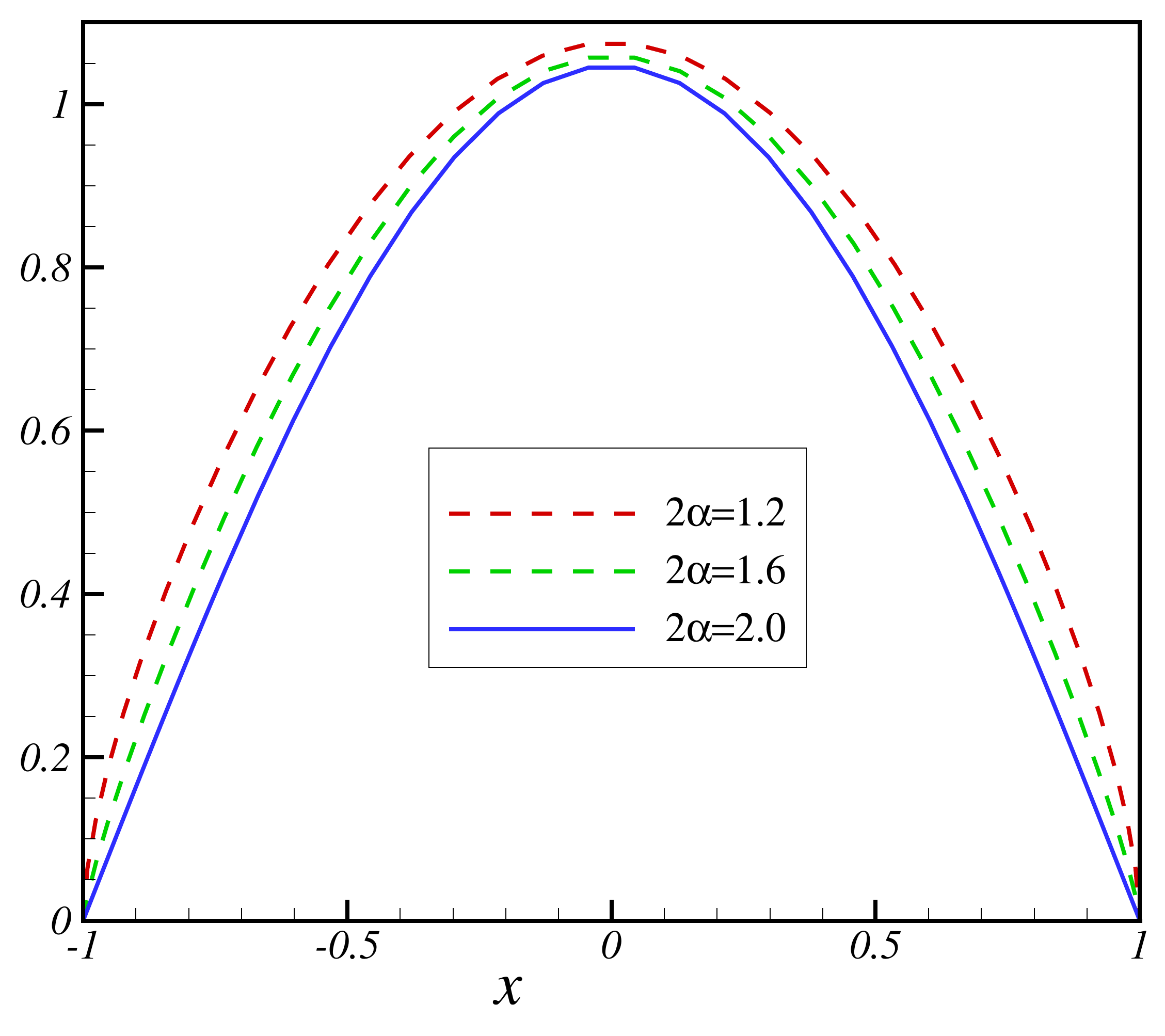}
\includegraphics[width=0.32\textwidth,angle=0]{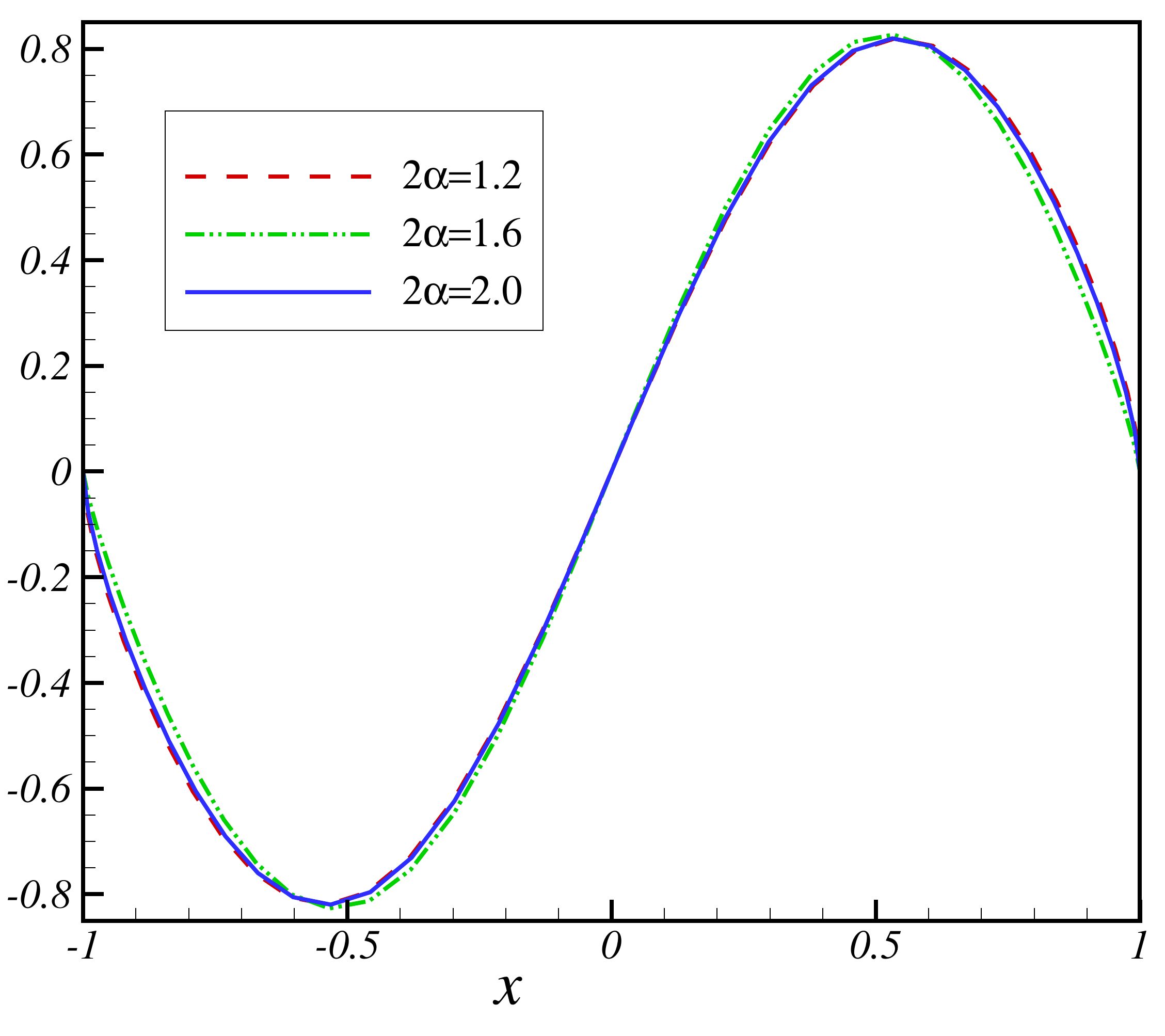}
\includegraphics[width=0.32\textwidth,angle=0]{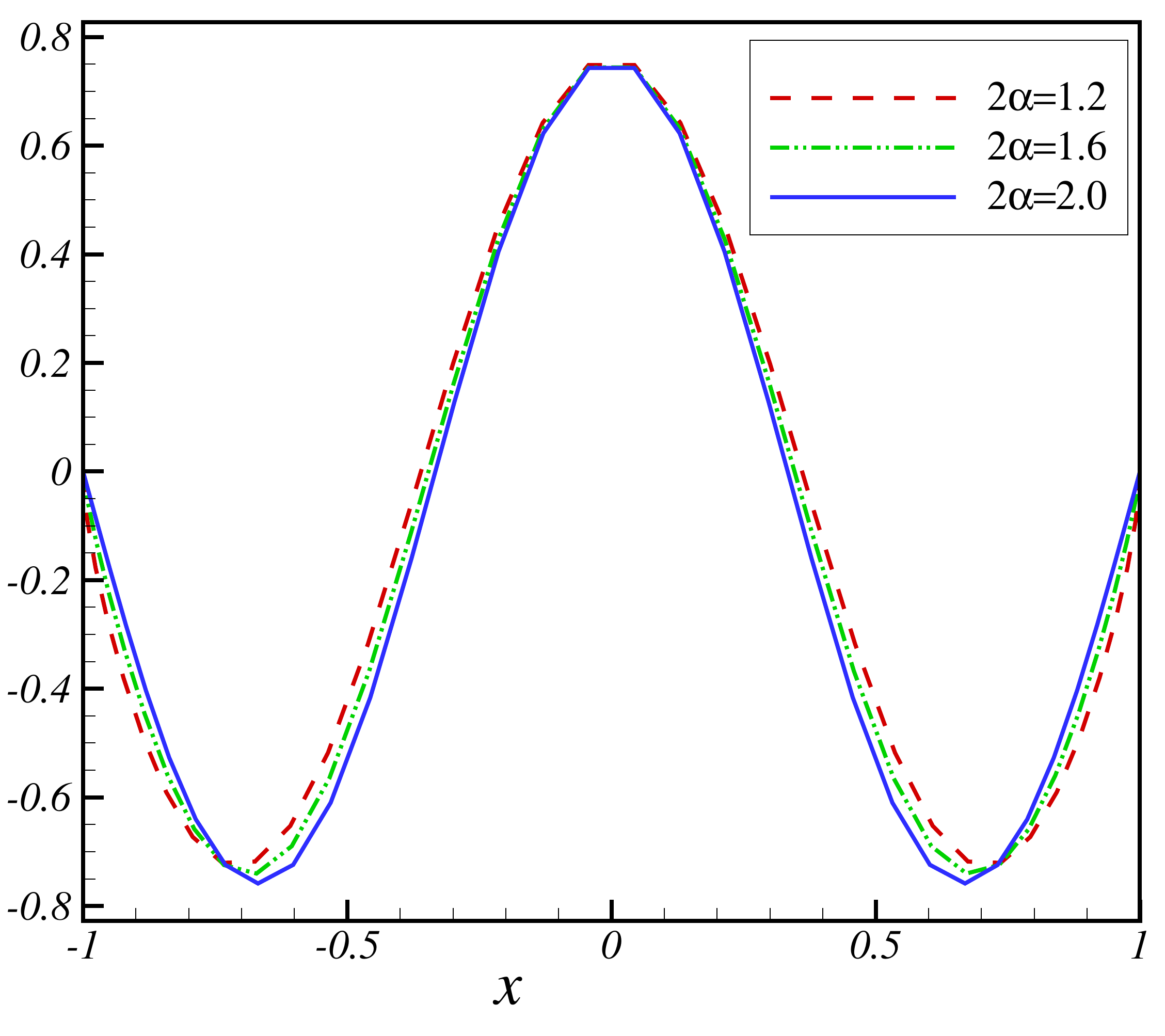}
\end{center}
\caption{The figure of the first three eigenvectors with different fractional
 order $2\alpha=1.2, 1.6, 2.0$ and $N=32$.}
\label{vector}
\end{figure}

\subsection{Weyl-type asymptotic law}

By a fixed computational parameter $N=1024$, we plot
all the eigenvalues with $N=1024$ for different fractional order
$2\alpha=1.2, 1.4, 1.6, 1.8, 2.0$  in Fig.\,\ref{eigallalp} (left). As expected,  it is about $\displaystyle \frac{2}{\pi}\approx 64\%$ of
total eigenvalues that are reliable and obey the Weyl-type asymptotic law. Then, the first 700 eigenvalues and function $n^{2\alpha}$
correspond to  $2\alpha=1.2$ and $2.0$ are  plotted in Fig.\,\ref{eigp}, which clearly shows the eigenvalues linear dependence of $n^{2\alpha}$.

Fig.\,\ref{eigallalp} (right)  presents   the 100-th, 200-th, 300-th, 400-th, 500-th, 600-th  eigenvalues versus $2\alpha=1.1$ to $2.0$ in
semi-logarithm scale. The straight lines indicate
that the valid eigenvalues increase algebraic with respect to the
fractional order parameter  $2\alpha$.

Finally, we plot in Fig.\,\ref{slope} the first 650 values $\rho_n=\displaystyle \frac{\lambda_n}{(\frac{\pi n}{2})^{2\alpha}}$
versus the $n$ for some $2\alpha$.
As we known, the exact eigenvalue $\lambda_n$ of $2\alpha=2$ is equal to $\displaystyle (\frac{\pi n}{2})^2$,
so the corresponding  $\rho_n$ is identical to $1$ which
was checked in Fig.\,\ref{slope}. Also, we can conclude that  $\lambda_n=\displaystyle (\frac{\pi n}{2})^{2\alpha}$
when $n$ is large enough by this figure. In order to better observe the  growth tendency of some leading  eigenvalues, we plot the
$\rho_n$ in the sub-figure of Fig.\,\ref{slope} in logarithm-logarithm scale. It shows that the $\rho_n$ slightly increases to 1
as $n$ increases for each $2\alpha$.

\begin{figure}[h!bt]
\hfill%
\includegraphics[width=0.35\textwidth,angle=0]{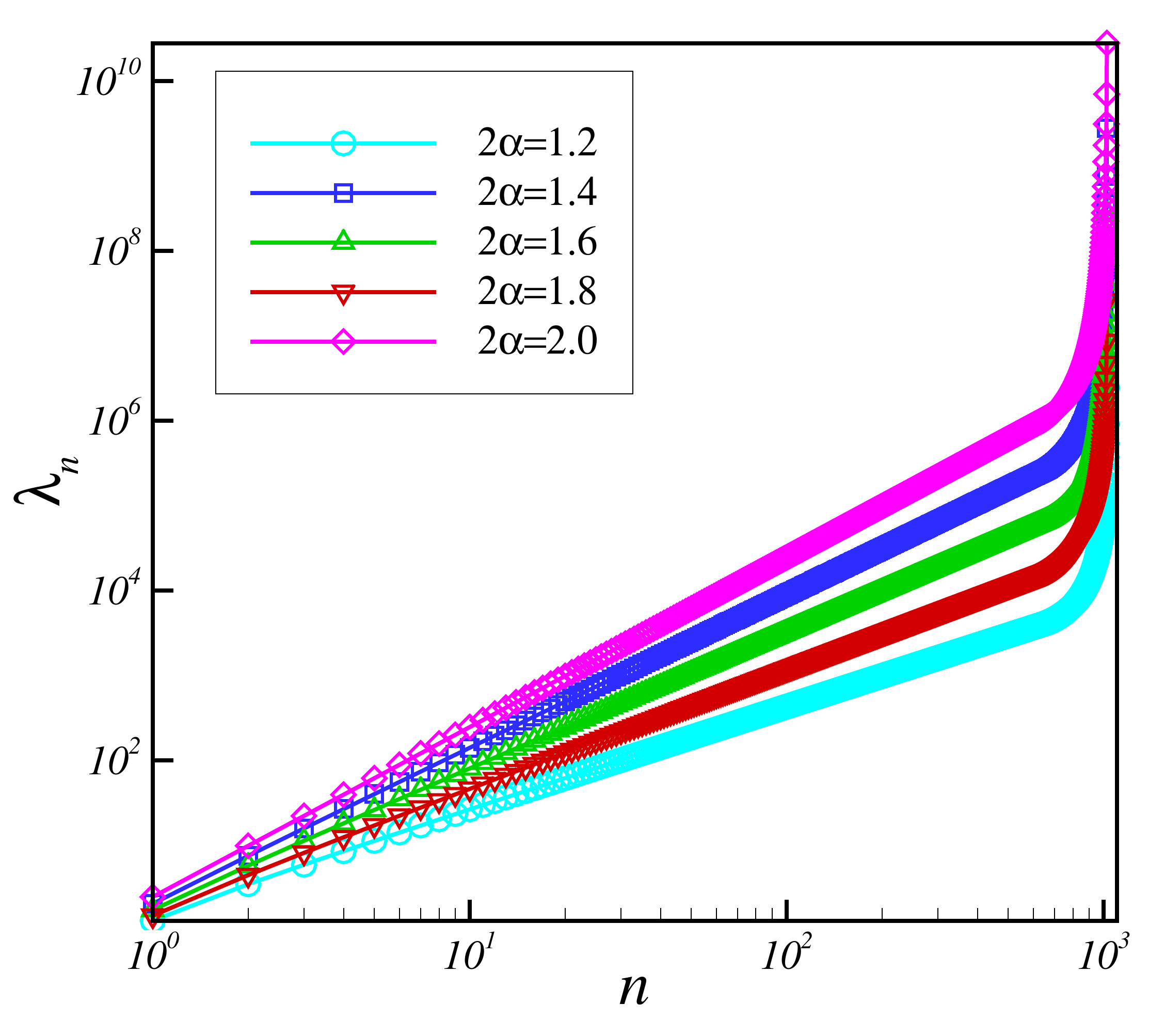}\hfill%
\includegraphics[width=0.35\textwidth,angle=0]{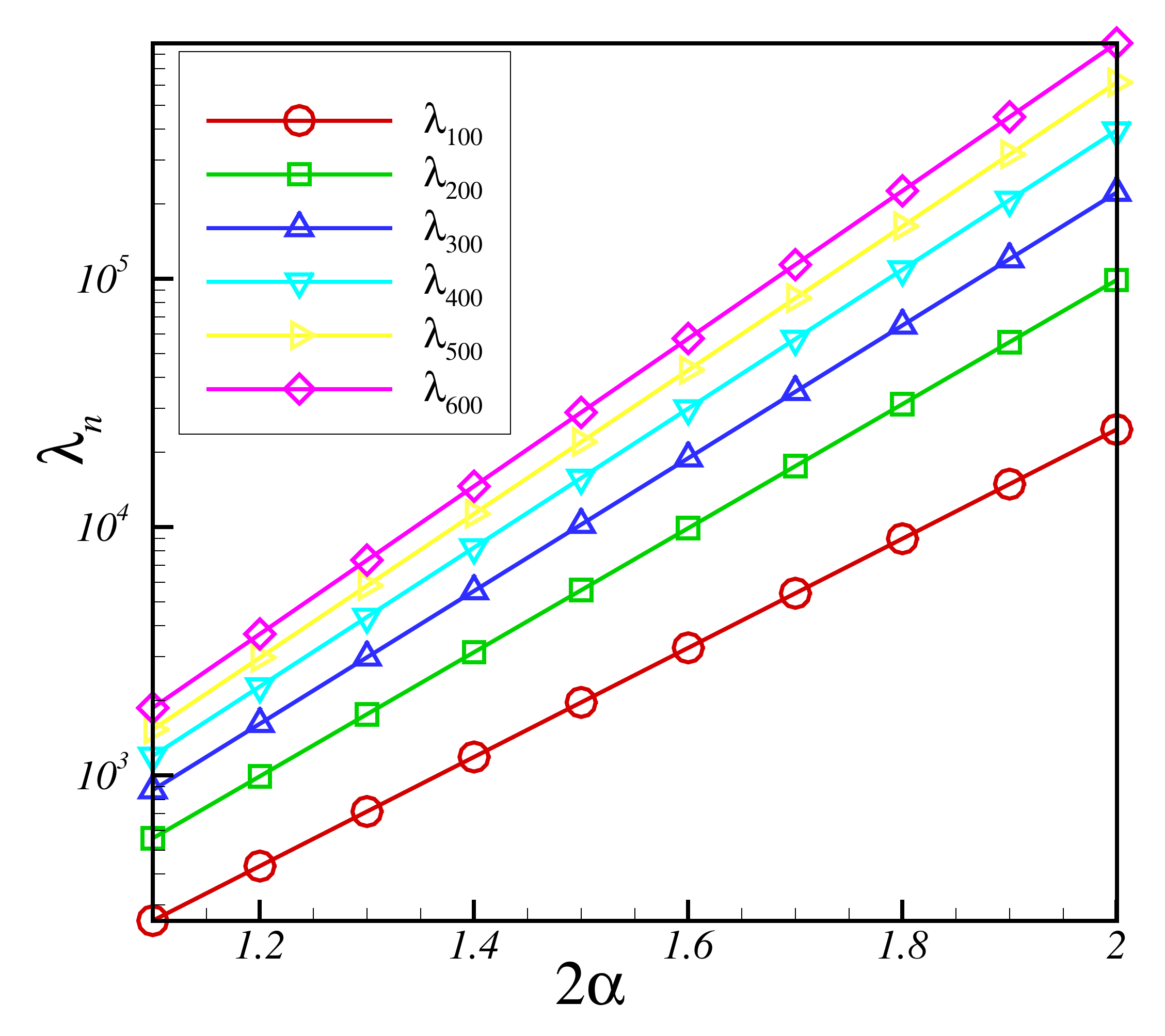}\hspace*{\fill}
\caption{Left: The all eigenvalues with different fractional  order $2\alpha$ from $1.2$ to $2.0$, $N=1024$. Right:
The $n$-th eigenvalue versus different fractional order $2\alpha$ from $1.1$ to $2.0$ in semi-logarithm scale.}
\label{eigallalp}
\end{figure}

\begin{figure}[http]
\hfill
\subfigure[$2\alpha=1.2$]{
\includegraphics[width=0.35\textwidth,angle=0]{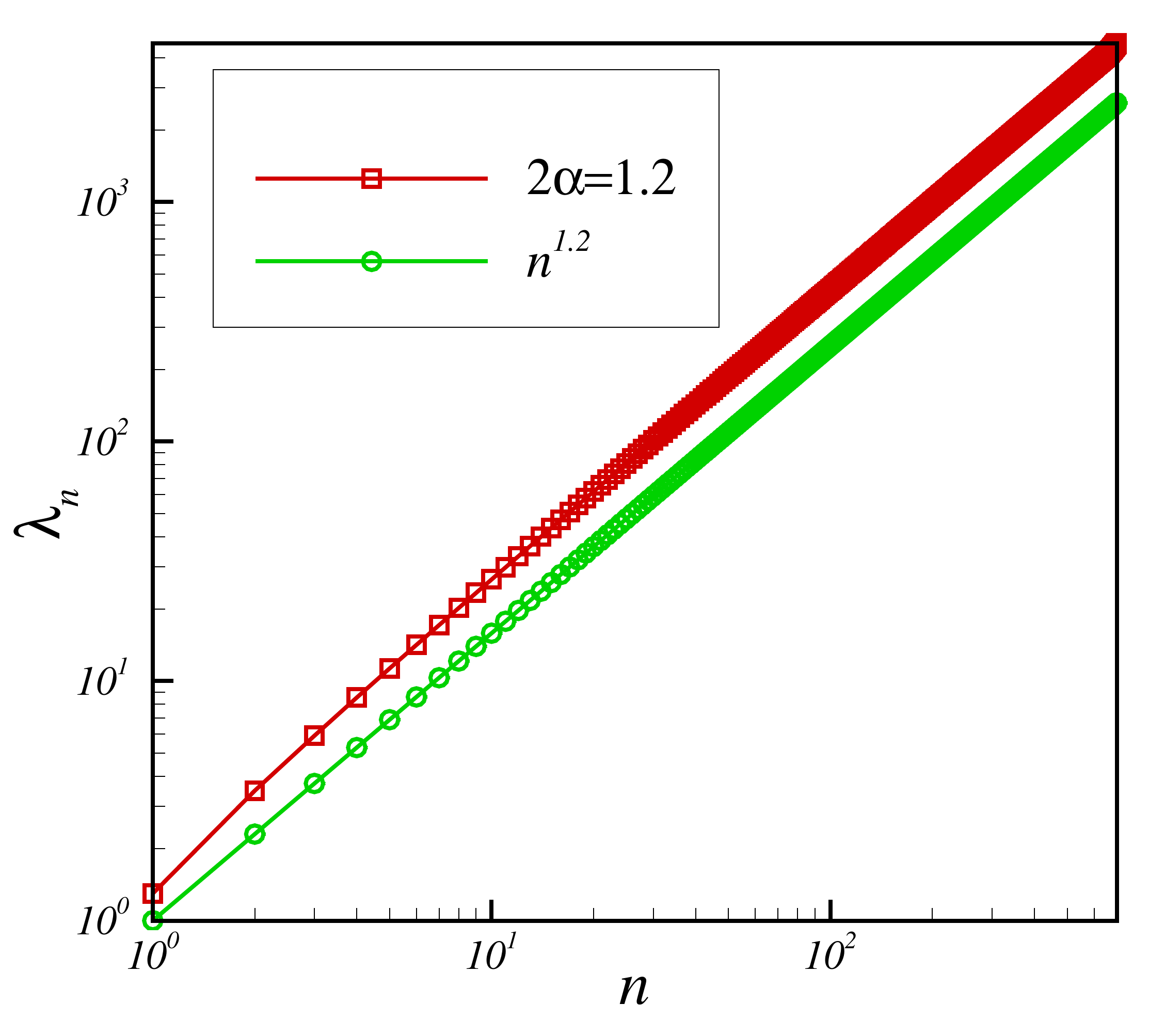}
}\hfill%
\subfigure[$2\alpha=1.8$]{
\includegraphics[width=0.35\textwidth,angle=0]{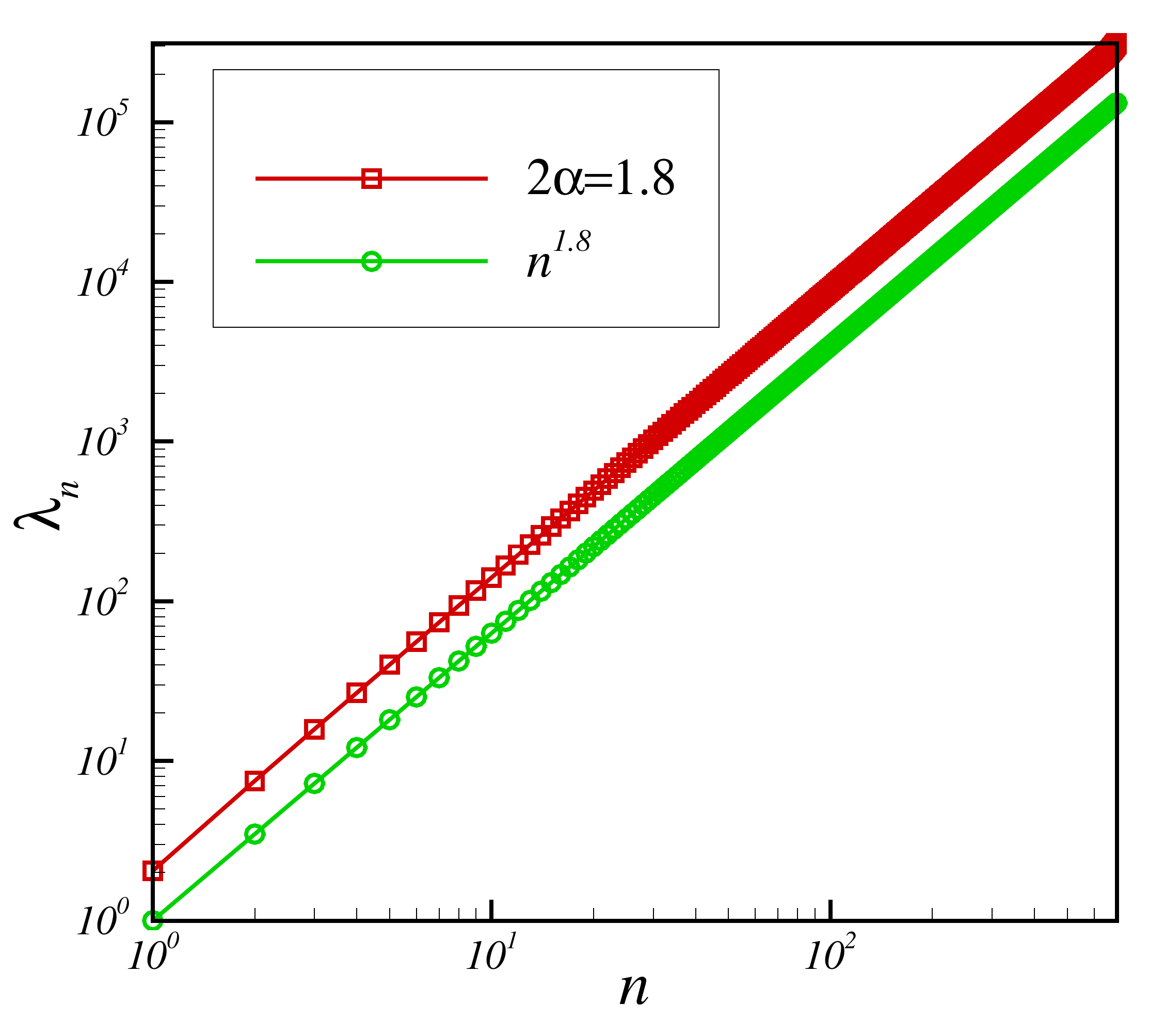}
}\hspace*{\fill}
\caption{The first 700  eigenvalues and function $n^{2\alpha}$ versus $2\alpha=1.2$ and $2.0$, with $N=1024$.}
\label{eigp}
\end{figure}

\begin{figure}[h!bt]
\begin{center}
\includegraphics[width=0.38\textwidth,,angle=0]{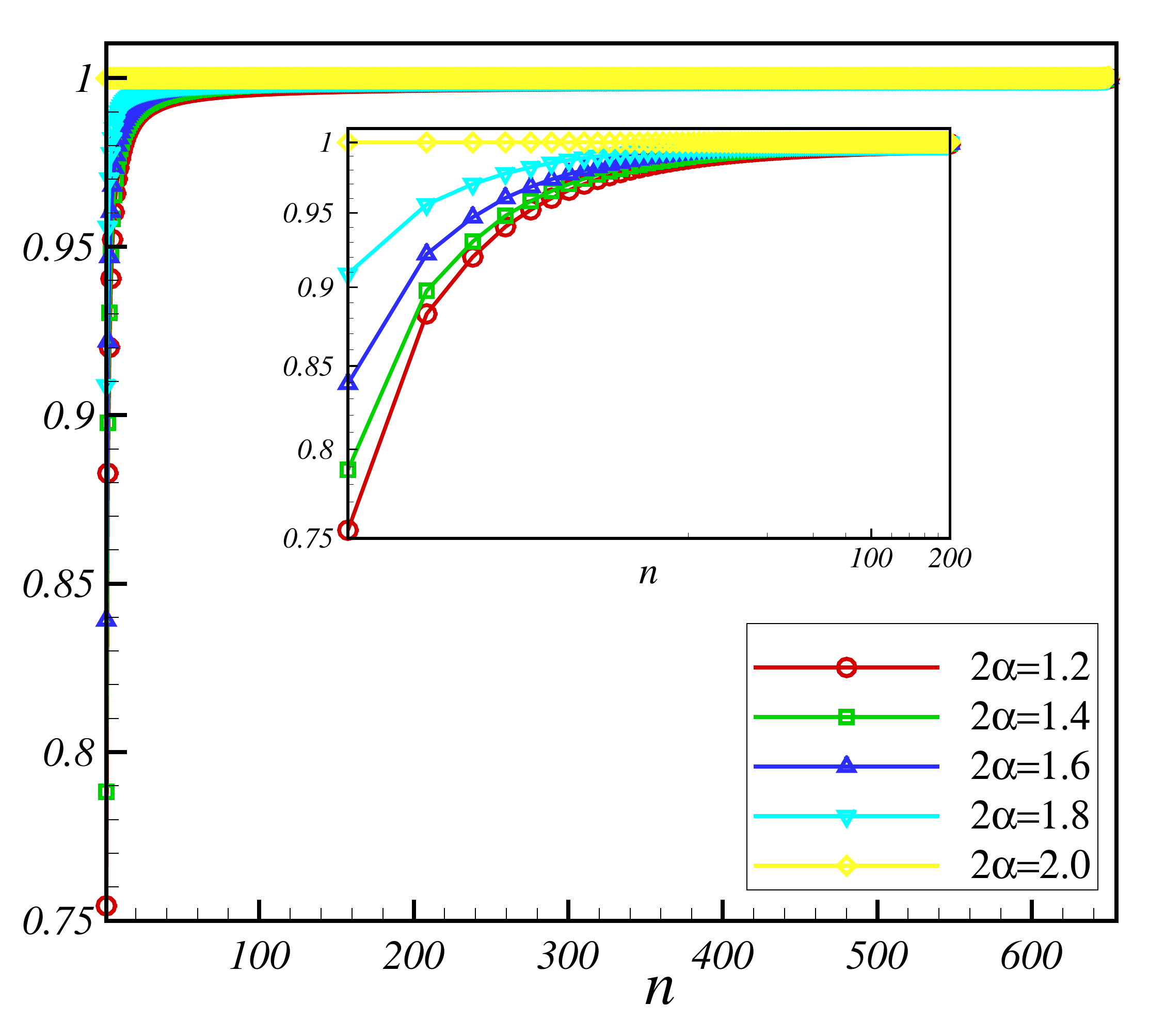}
\end{center}
\caption{The first 650 values $\rho_n=\displaystyle \frac{\lambda_n}{(\frac{\pi n}{2})^{2\alpha}}$ versus  $n$ for different fractional
partial order $2\alpha$  from $1.2$ to $2.0$. The sub-figure plots the first 200 values of $\rho_n$ versus  $n$  in logarithm-logarithm scale.}
\label{slope}
\end{figure}

\subsection{Condition number}

In this subsection, we plot in Fig.\,\ref{cond} the condition number versus polynomial degree $N$ with different fractional order $2\alpha$  from $1.2$ to $2.0$ in logarithm-logarithm scale. The result presented in Fig.\,\ref{cond} shows that, for each $2\alpha$, the condition number $\chi_N$ grows algebraically with respect to $N$. In order to investigate the growth tendency of the condition number numerically, we plot the condition number together with the function $N^{4\alpha}$ in logarithm-logarithm scale with fixed $2\alpha=1.2$ and $1.8$ in the left and right of Fig.\,\ref{condalp} respectively.  The straight lines
are evidence of $\chi_N=\mathcal{O}(N^{4\alpha})$ which was predicted by Theorem \ref{COND}.

\begin{figure}[h!bt]
\begin{center}
\includegraphics[width=0.38\textwidth,angle=0]{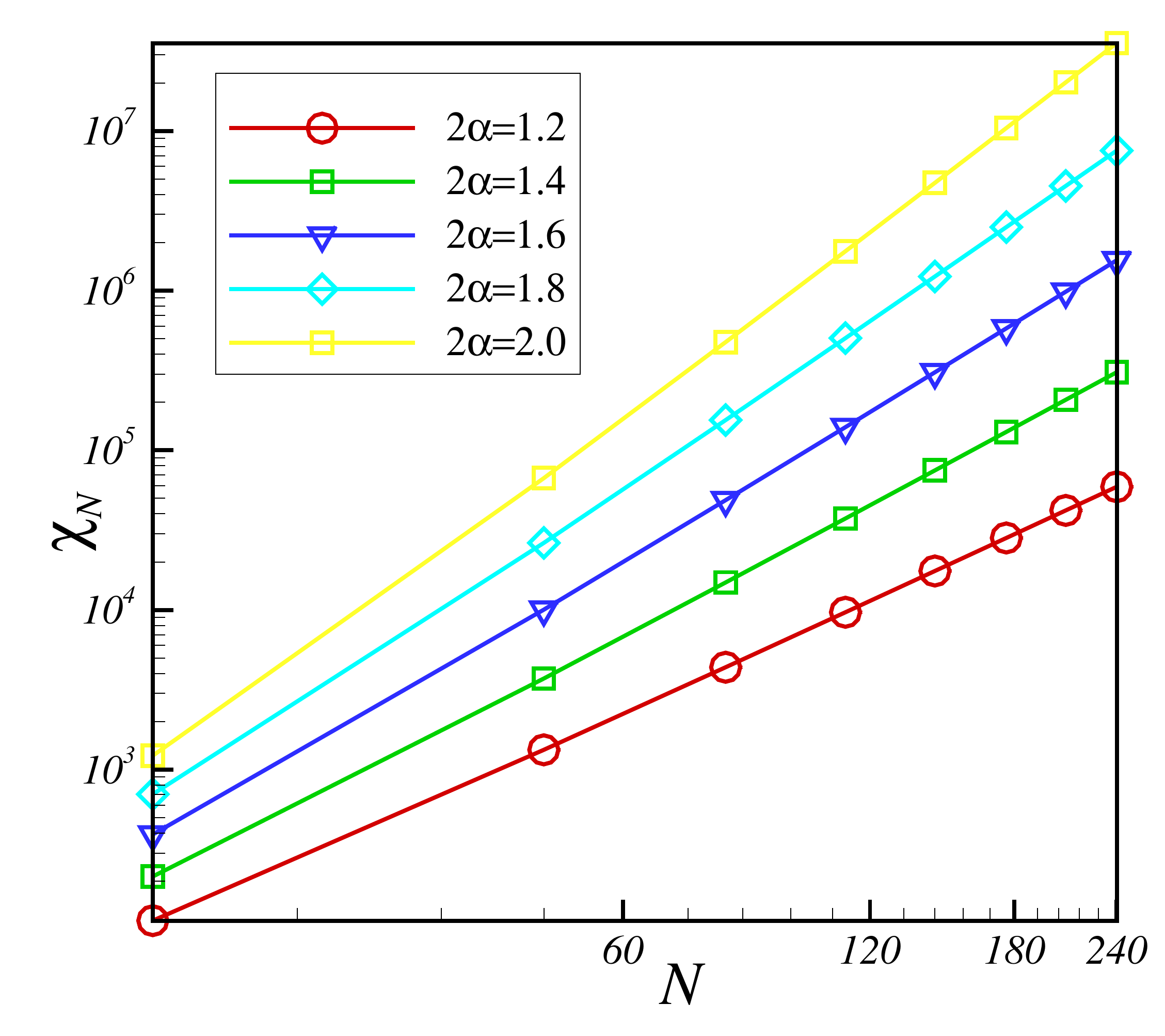}
\end{center}
\caption{The condition number versus the polynomial degree $N$ with different fractional
 order $2\alpha$  from $1.2$ to $2.0$ in logarithm-logarithm scale.}
\label{cond}
\end{figure}
\begin{figure}[http]
\hfill
\subfigure[$2\alpha=1.2$]{
\includegraphics[width=0.35\textwidth,angle=0]{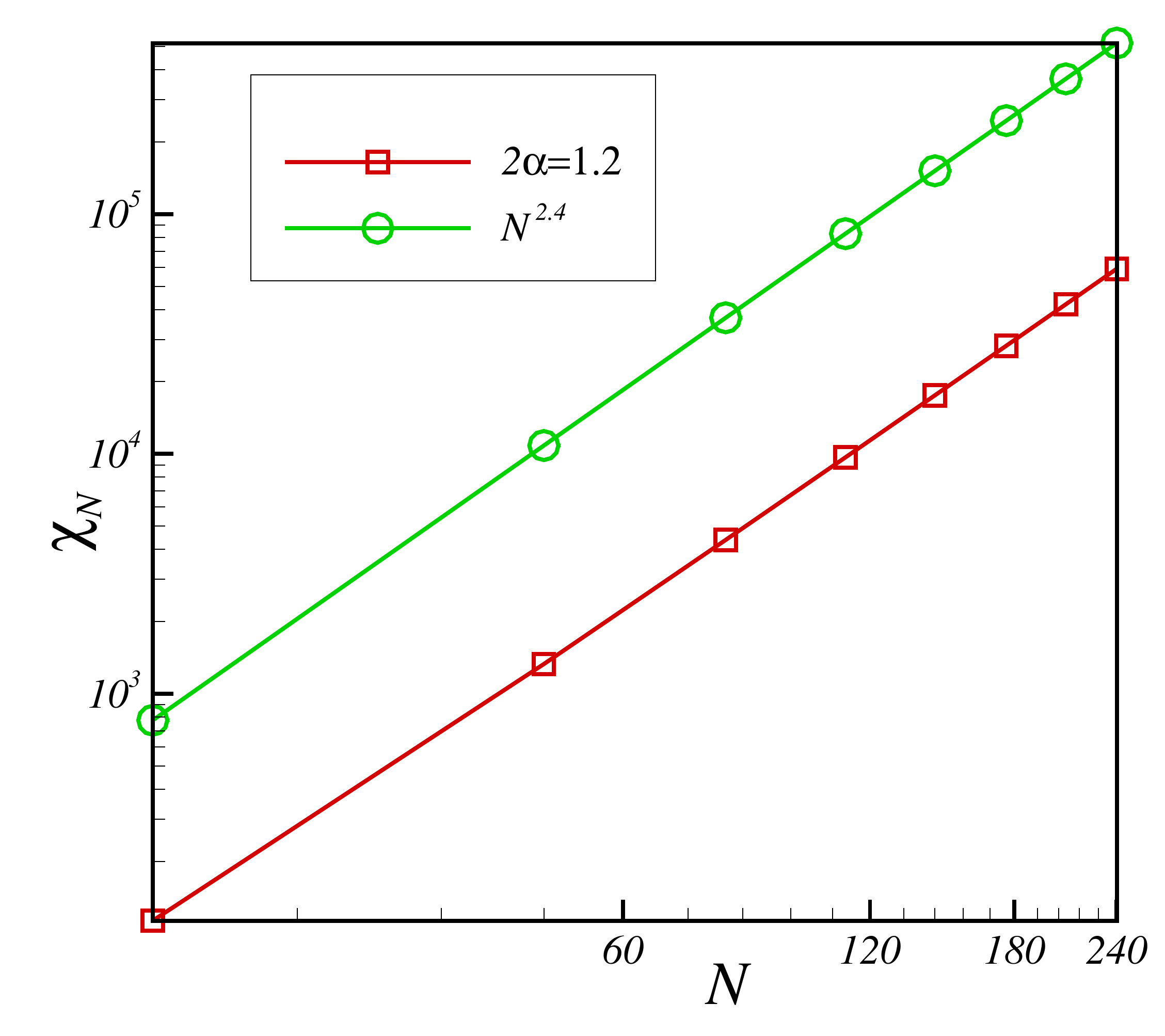}
}\hfill
\subfigure[$2\alpha=1.8$]
{
\includegraphics[width=0.35\textwidth,angle=0]{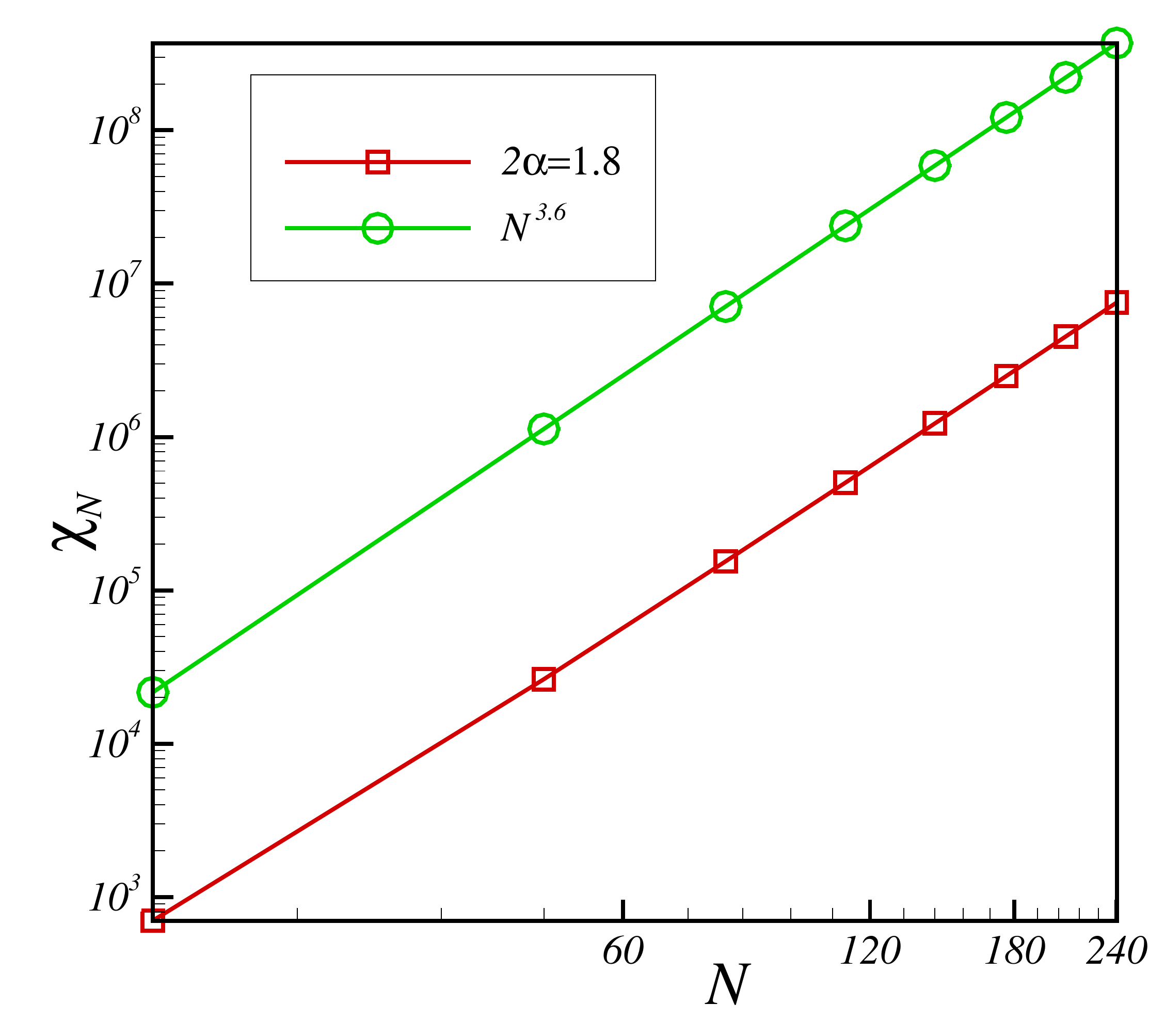}
}\hspace*{\fill}
\caption{The condition number and polynomial function $N^{4\alpha}$ versus the polynomial degree $N$ with different fractional
 order $2\alpha=1.2$ and $2\alpha=1.8$ in logarithm-logarithm scale.}
\label{condalp}
\end{figure}

\begin{example}
Higher fractional derivative order $2\alpha$
\end{example}

In this example, we  repeat the tests of the eigenvalue problems with more higher
fractional derivative order $2\alpha$. We plot the accuracy, eigenvalues and condition number for $2\alpha=3.6$
and $5.6$ in Figs.\,\ref{erroreig1}, \ref{eigacond} (left), \ref{eigacond} (right) and \ref{condalp1} respectively.
This numerical results demonstrate that  the Jacobi-Galerkin spectral method is still accurate and efficient with larger $2\alpha$.

\begin{figure}[h!bt]
\begin{center}
\includegraphics[width=0.32\textwidth,angle=0]{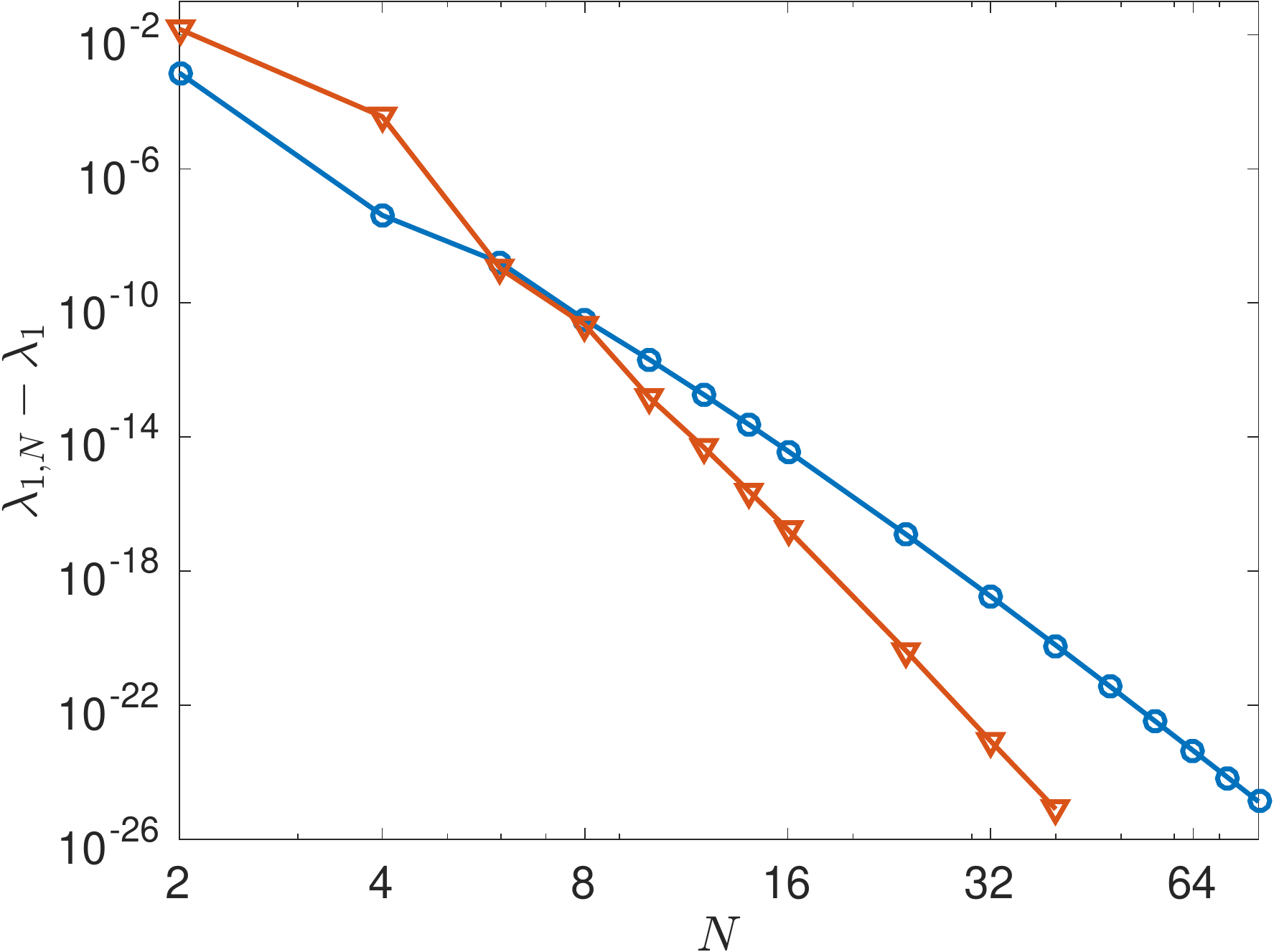}
\includegraphics[width=0.32\textwidth,angle=0]{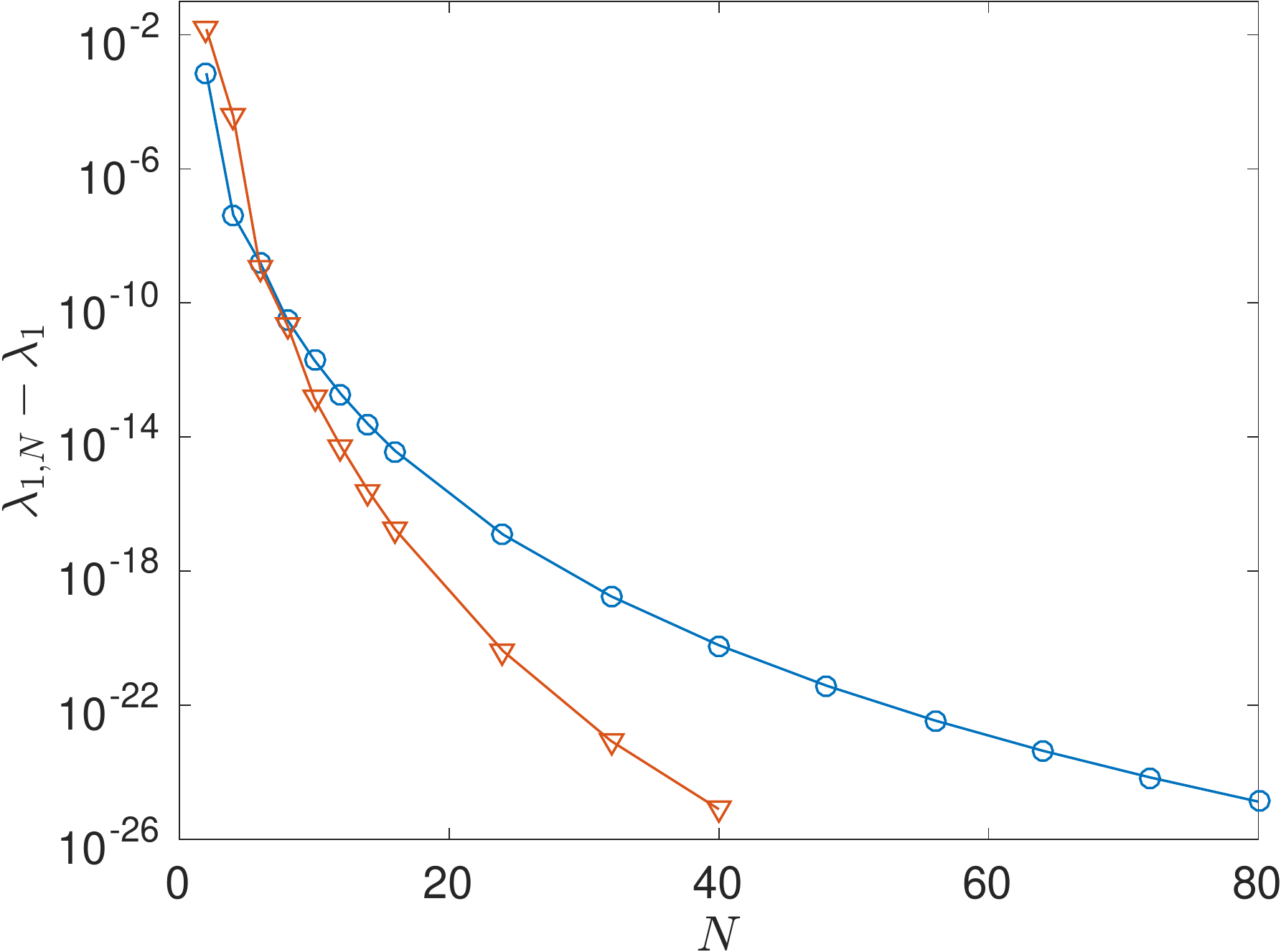}
\includegraphics[width=0.3\textwidth,angle=0]{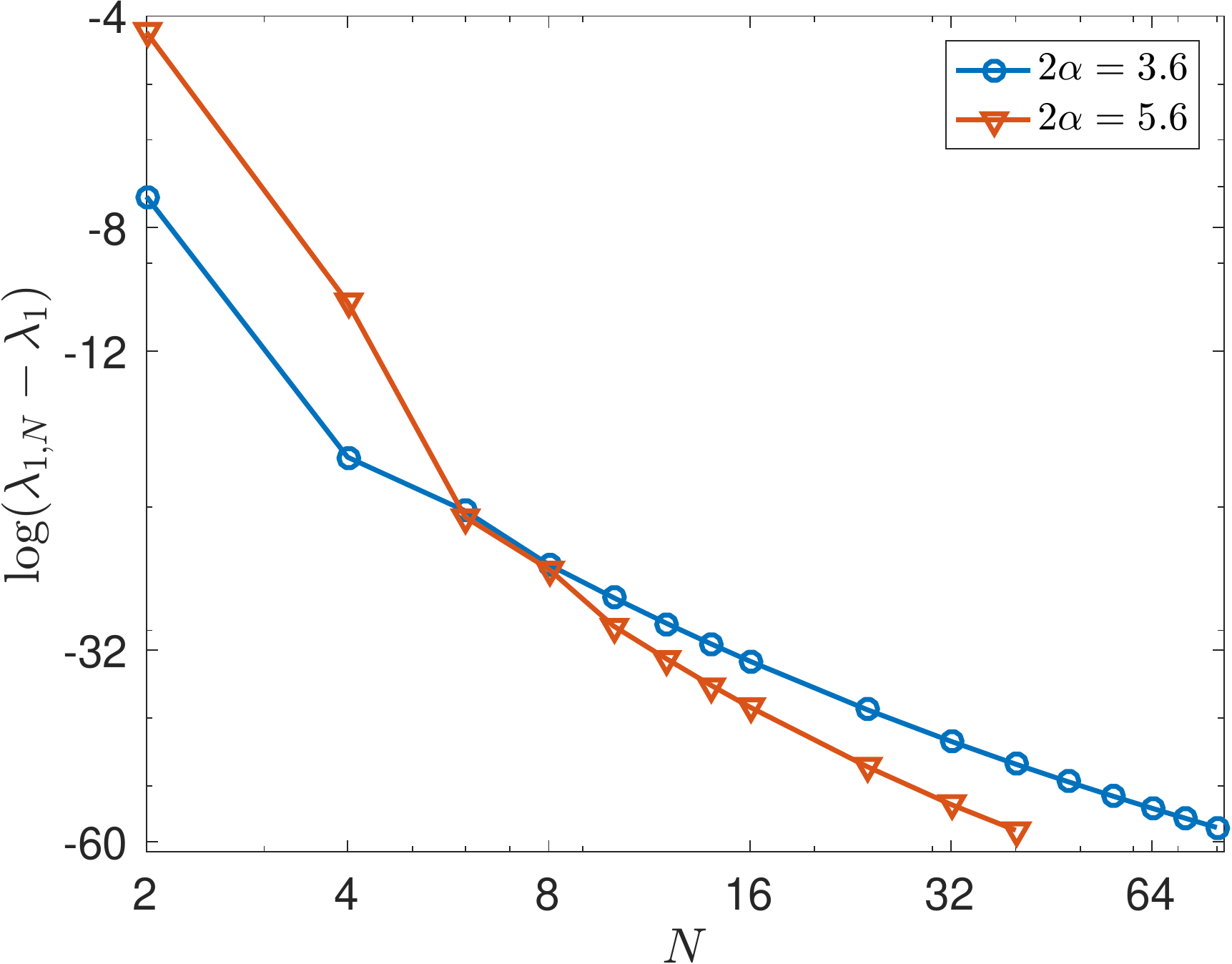}
\end{center}
\caption{Errors of the first eigenvalue versus polynomial degree $N$  with various fractional orders. Left: $\lambda_{1,N}-\lambda_1$ in logarithm-logarithm scale; Center:  $\lambda_{1,N}-\lambda_1$  in semi-logarithm scale; Right: 
$\log(\lambda_{1,N}-\lambda_1)$ in logarithm-logarithm scale.}
\label{erroreig1}
\end{figure}

\begin{figure}[h!bt]
\hfill
\includegraphics[width=0.32\textwidth,angle=0]{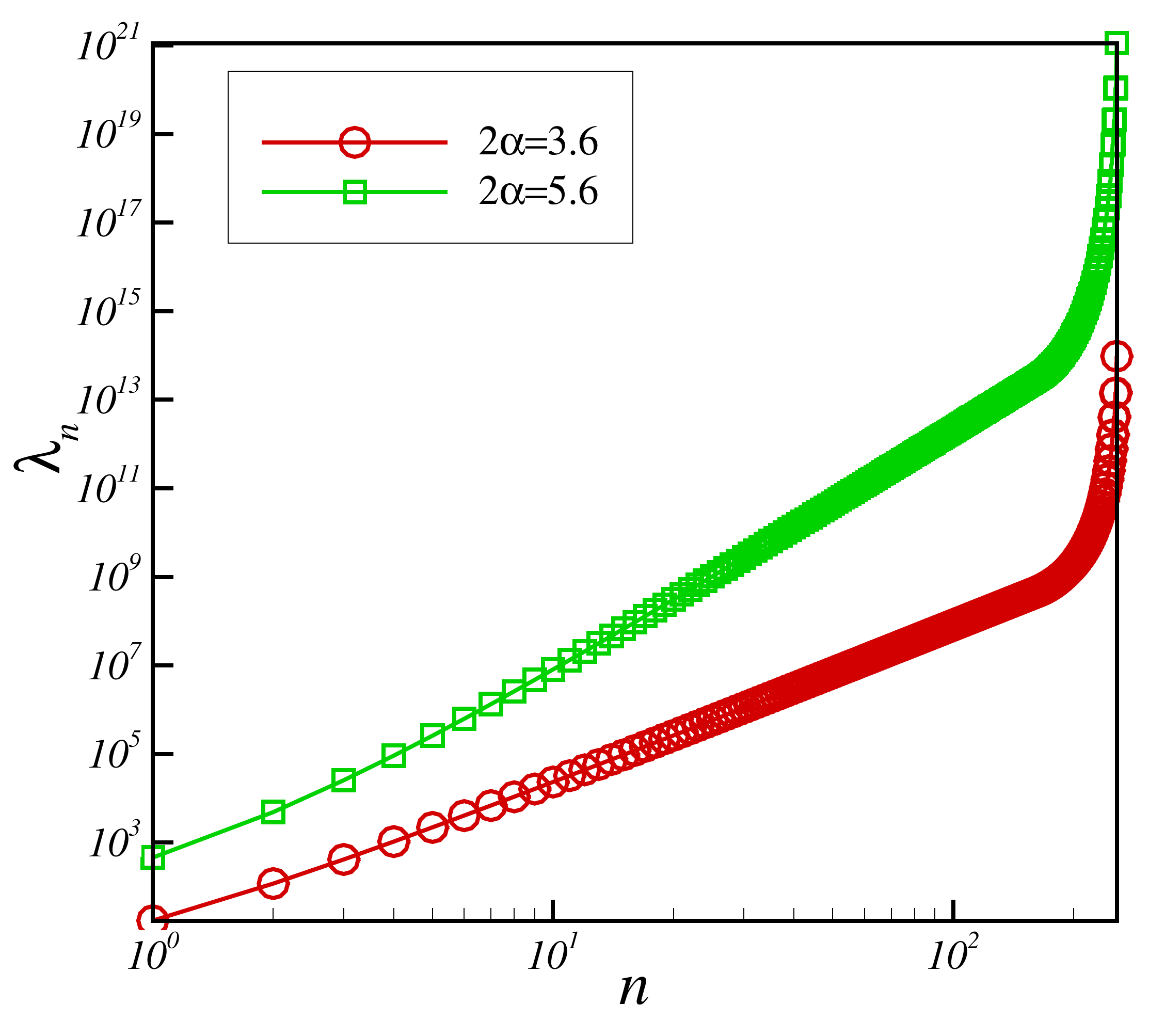}
\hfill
\includegraphics[width=0.32\textwidth,angle=0]{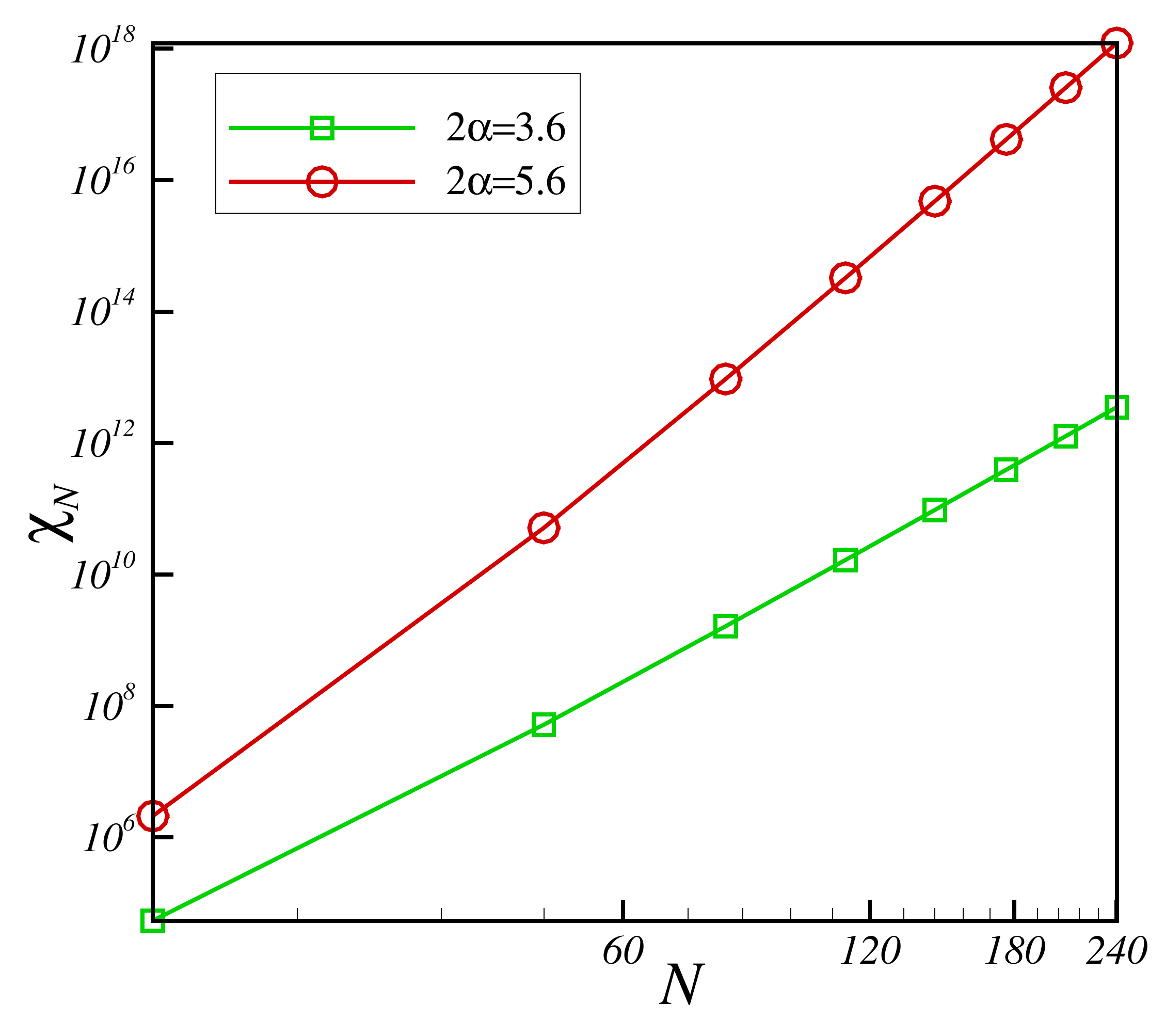}
\hspace*{\fill}
\caption{Left: The all eigenvalues with different fractional  order $2\alpha$ equal to $3.6$ and $5.6$, $N=256$.
Right: The condition number versus the polynomial degree $N$ with different fractional
 order $2\alpha=3.6$ and $2\alpha=5.6$ in logarithm-logarithm scale.}
\label{eigacond}
\end{figure}

\begin{figure}[http]
\hfill
\subfigure[$2\alpha=3.6$]{
\includegraphics[width=0.35\textwidth,angle=0]{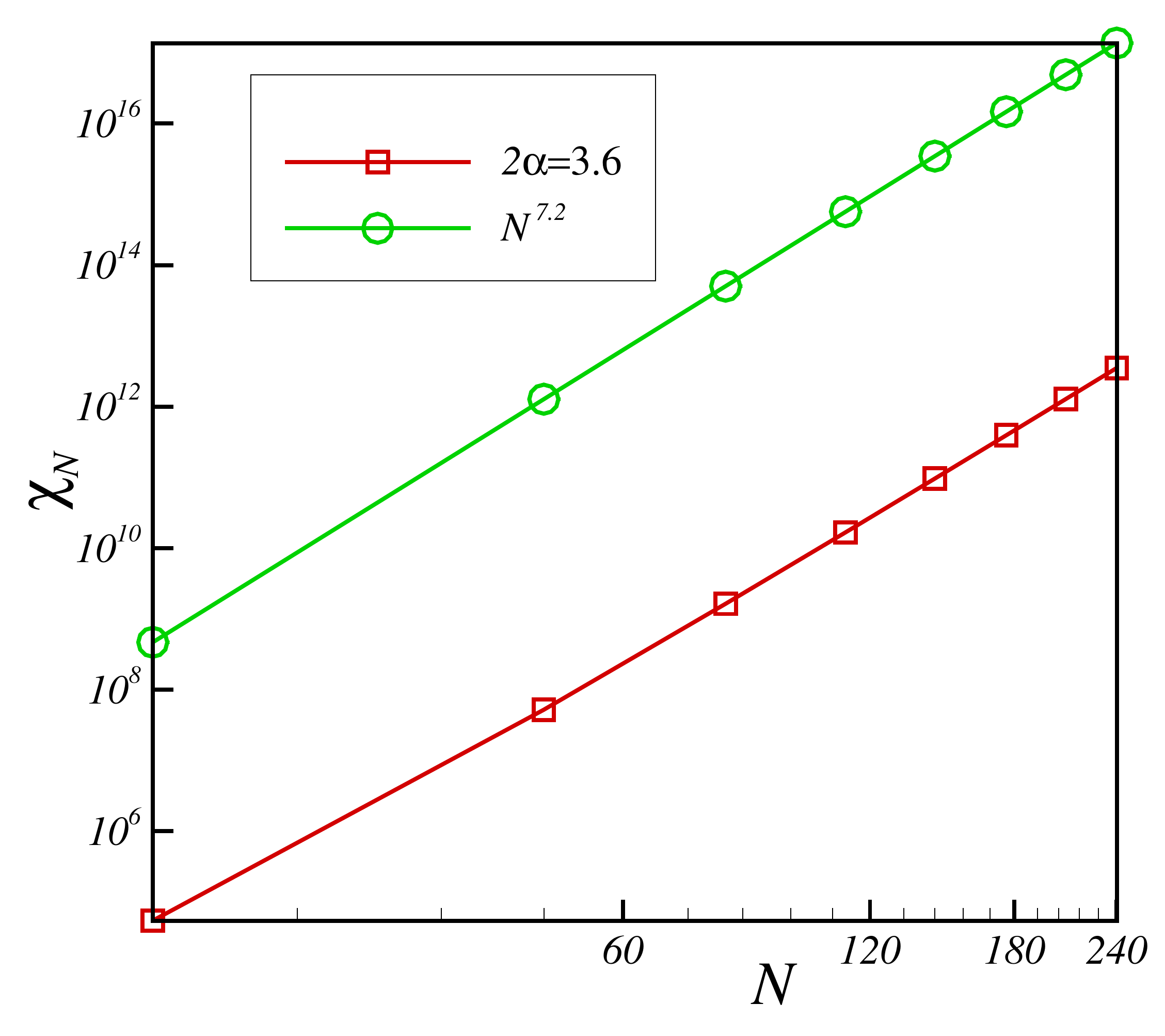}
}\hfill
\subfigure[$2\alpha=5.6$]
{
\includegraphics[width=0.35\textwidth,angle=0]{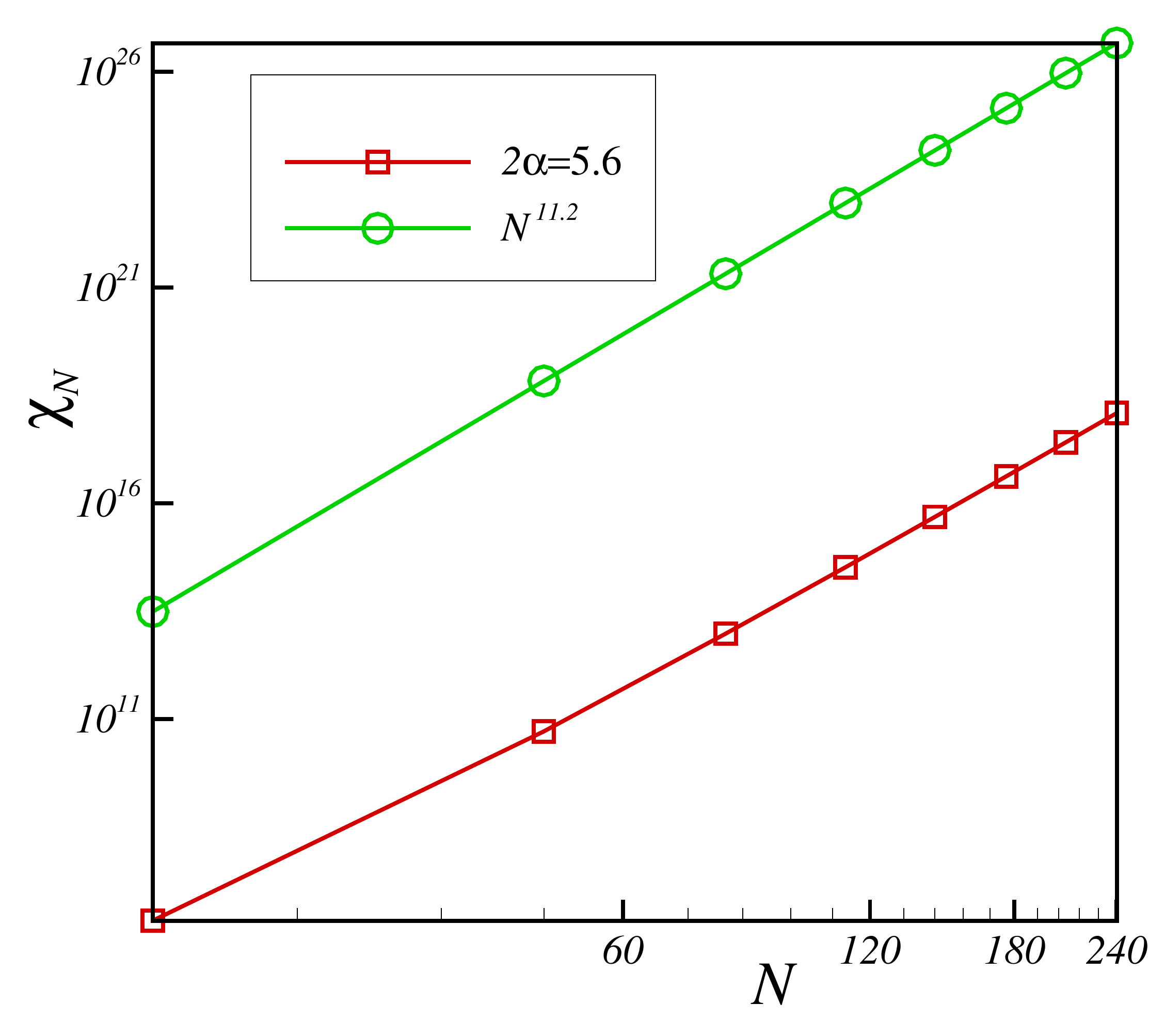}
}\hspace*{\fill}
\caption{The condition number and polynomial function $N^{4\alpha}$ versus the polynomial degree $N$ with different fractional
 order $2\alpha=3.6$ and $2\alpha=5.6$ in logarithm-logarithm scale.}
\label{condalp1}
\end{figure}

\section{Conclusion}
\label{sec:finale}
In this paper, we have proposed an efficient Jacobi-Galerkin spectral method for the
Riesz fractional differential eigenvalue problems. The error estimate for the
eigenvalues  was derived by the compact spectral theory. Numerical
results confirm  the  asymptotically exponential convergence rate of the eigenvalues. Also
the eigenvalues with different fractional partial differential order were calculated.
The figures demonstrate that the eigenvalues behavior as $n^{2\alpha}$ which obey the
Weyl-type asymptotic law.
In the future, we will consider the fractional differential eigenvalue problems
in two dimensional case and more complex computational domain.


\bibliographystyle{plain}
\bibliography{ref}

\end{document}